\documentclass{amsart}
 \usepackage{color}
 \usepackage{amssymb}
\usepackage[utf8]{inputenc}
 \usepackage{soul}

\newtheorem{theorem}{Theorem}
\newtheorem{lemma}{Lemma}
\newtheorem{proposition}{Proposition}
\newtheorem{corollary}{Corollary}

\theoremstyle{definition}

\theoremstyle{remark}
\newtheorem{remark}{Remark}

\def \a{\alpha }
\def \b {\beta}

\newcommand{\R}[1]{\mathcal R^{({#1})}}
\newcommand{\V}[1]{\mathcal V^{({#1})}}
\newcommand{\A}[1]{\mathcal A^{({#1})}}
\newcommand{\B}[1]{\mathcal B^{({#1})}}
\newcommand{\W}[1]{\mathcal W^{({#1})}}
\newcommand{\sltwo}{\mathfrak{sl}_2}
\newcommand{\asltwo}{\widehat{\sltwo}}
\newcommand{\lksl}{L_k(\sltwo)}
\newcommand{\PiO}{\Phi(0)}
\newcommand{\LL}[2]{\mathcal L_{#1}^{(#2)}}

\newcommand{\bea}{\begin{eqnarray}}
\newcommand{\eea}{\end{eqnarray}}
\newcommand{\be}{\begin {equation}}
\newcommand{\ee}{\end{equation}}

\newcommand{\Z}{\Bbb Z}

\newcommand{\C}{\Bbb C}
\newcommand{\Q}{\Bbb Q}

\newcommand{\la}{\langle}
\newcommand{\ra}{\rangle}

\begin{document}

\title{ The vertex algebras $\mathcal R^{(p)}$ and $\V{p}$  }

\author[] {Dra\v zen Adamovi\' c,
Thomas Creutzig,
Naoki Genra and
Jinwei Yang}

\begin{abstract}
The vertex algebras $\mathcal V^{(p)}$ and $\mathcal R^{(p)}$ introduced in \cite{A} are very interesting relatives of the famous triplet algebras of logarithmic CFT.
The algebra $\mathcal V^{(p)}$ (respectively, $\mathcal R^{(p)}$) is a large extension of the simple affine vertex algebra $\lksl$ (respectively, $\lksl$ times a Heisenberg algebra), at level $k=-2+1/p$ for positive integer $p$. Particularly, the algebra $\mathcal V^{(2)}$ is the simple small $N=4$ superconformal vertex algebra with $c=-9$, and $\mathcal R^{(2)}$ is $L_{-3/2}(\mathfrak{sl}_3)$. In this paper, we derive structural results of these algebras and prove various conjectures coming from representation theory and physics.

We show that $SU(2)$ acts as automorphisms on $\V{p}$ and
we decompose $\V{p}$ as an $\lksl$-module and $\R{p}$ as an $L_k(\mathfrak{gl}_2)$-module. The decomposition of $\V{p}$ shows that $\V{p}$ is the large level limit of a corner vertex algebra appearing in the context of $S$-duality.
We also show that the quantum Hamiltonian reduction of $\V{p}$ is the logarithmic doublet algebra $\A{p}$ introduced in \cite{AdM-doublet}, while the reduction of $\R{p}$ yields the $\B{p}$-algebra of \cite{CRW}. Conversely, we realize $\V{p}$ and $\R{p}$ from $\A{p}$ and $\B{p}$ via a procedure that deserves to be called inverse quantum Hamiltonian reduction.
As a corollary, we obtain that the category $KL_{k}$ of ordinary $\lksl$-modules at level $k=-2+1/p$ is a rigid vertex tensor category equivalent to a twist of the category $\text{Rep}(SU(2))$. This finally completes rigid braided tensor category structures for $\lksl$ at all complex levels $k$.

We also establish a uniqueness result of certain vertex operator algebra extensions and use this result to prove that both $\R{p}$ and $\B{p}$ are certain non-principal $\mathcal W$-algebras of type $A$ at boundary admissible levels. The same uniqueness result also shows that $\R{p}$ and $\B{p}$ are the chiral algebras of Argyres-Douglas theories of type $(A_1, D_{2p})$ and $(A_1, A_{2p-3})$.
\end{abstract}

\maketitle

\date{}


\bibliographystyle{amsalpha}

\theoremstyle{plain}
\newtheorem*{introthm}{Theorem}
\newtheorem{obs}{Observation}
\newtheorem{thm}{Theorem}[section]
\newtheorem{prop}[thm]{Proposition}
\newtheorem{lem}[thm]{Lemma}
\newtheorem{cor}[thm]{Corollary}
\newtheorem{conj}[thm]{Conjecture}

\theoremstyle{definition}
\newtheorem{defi}[thm]{Definition}
\newtheorem{rem}[thm]{Remark}

\newcommand {\AL}{AL}
\newcommand {\CC}{\mathbb{C}}
\newcommand {\ZZ}{\mathbb{Z}}
\newcommand {\tr}{\text{tr}}
\newcommand {\ch}{\text{ch}}
\newcommand {\sch}{\text{sch}}
\newcommand {\cW}{\mathcal{W}}
\newcommand {\cX}{\mathcal{X}}
\newcommand {\cB}{\mathcal{B}}
\newcommand {\cS}{\mathcal{S}}
\newcommand {\cO}{\mathcal{O}_p}
\newcommand {\voa}{vertex operator algebra}
\newcommand {\voas}{vertex operator algebras}
\newcommand{\Sing}{M(p)}
\newcommand{\Trip}{W(p)}
\newcommand{\vak}{\bf 1}

\newcommand{\hopflink}{{\text{\textmarried}}}

\renewcommand{\baselinestretch}{1.2}

\maketitle

\section{Introduction}

The singlet \cite{A-singlet, AdM-singlet, CM, CMW}, doublet \cite{AdM-doublet}, triplet \cite{AdM-triplet, AdM-zhu, FGST, FGST2, TW},   and $\B{p}$ \cite{CRW, ACKR} algebras   are the best understood examples of \voas{} with non semi-simple represenration theory and they are of significant importance for logarithmic conformal field theory \cite{CR-log, CG, AdM-LCFT}.    These algebras are large extensions of the Virasoro \voa{} $L^{\rm{Vir}}(c_{1, p},0)$ at central charge $c_{1, p} = 1 - 6(p-1)^2/p$ for $p$ in $\Z_{\geq 2}$. The Virasoro algebra in turn is the quantum Hamiltonian reduction of the affine \voa{} $\lksl$ of $\sltwo$ at level $k=-2+\frac{1}{p}$. In this work, we realize and study \voas{} whose quantum Hamiltonian reductions are these well-known singlet, triplet and $\B{p}$-algebras. These algebras provide important sources of logarithmic conformal field theories and we will investigate their representation theory in future work. The importance of the present work is to resolve various open questions motivated from four dimensional physics, i.e. questions in Argyres-Douglas theories and in $S$-duality. Along the way, we discover a few additional interesting structure, which we shall describe in detail. First we introduce the definitions and the main properties of these important algebras.

\subsection{The $\V{p}$-algebra}

The $\V{p}$-algebra introduced in \cite{A} is a certain abelian intertwining algebra that we shall study first. Let us briefly recall its definition (see \eqref{defVp} for the detail): The $\V{p}$-algebra is a subalgebra of $M \otimes F_{\frac{p}{2}}$ where $M$ is the Weyl \voa{} (also often called the $\beta\gamma$-system) and
$F_{\frac{p}{2}}$ is the abelain intertwining algebra associated to the weight lattice of $\sltwo$ rescaled by $\sqrt{p}$. It is characterized as the kernel of a screening operator $\widetilde Q$ (see (\ref{scr-gen})):
\[
\V{p} = \mbox{Ker}  _{ M \otimes F_{\frac{p}{2}} } \widetilde{Q}.
\]

We think of $\V{p}$ as an analogue of the doublet algebra $\A{p}$ introduced in \cite{AdM-doublet}. The doublet algebra is an abelian intertwining algebra with $SL(2, \C)$ acting as automorphisms \cite{ALM} and it is a large extension of the Virasoro algebra at central charge $c_{1, p}$. Its even subalgebra is the famous triplet algebra. We elaborate various relations between $\A{p}$ and $\V{p}$. Firstly, our
Corollary \ref{cor;VpviaAp} says that $\V{p}$ is a subalgebra of $\A{p} \otimes \Pi(0)^\frac{1}{2}$, where $\Pi(0)^\frac{1}{2}$ is a certain extension along a rank one isotropic lattice of a rank two Heisenberg \voa. $\V{p}$ is then characterized as the kernel of another screening operator $S$ (see (\ref{operatorS})):
\[
 \V{p} = \mbox{\rm{Ker}}_{\A{p} \otimes \Pi(0)^{\frac{1}{2}} } S.
\]

Set $k=-2+\frac{1}{p}$ and $p \in \Z_{\geq 1}$. We denote by $\LL{n}{p}$ the simple highest-weight module of $\lksl$ of highest-weight $n\omega$ with $\omega$ the fundamental weight of $\sltwo$. We also use the short-hand notation $\rho_n=\rho_{n\omega}$ for the integrable $\sltwo$-modules. One of our main aims was to prove

\begin{theorem} \label{thm;sl2decomp} \textup{(Theorem \ref{thm:Vp-dec} and Corollary \ref{sl2der})} \newline
The Lie algebra $\sltwo$ acts on $\V{p}$ as derivations and $SL(2, \C)$ is a group of automorphisms. Moreover,
$\V{p}$ decomposes as an $\sltwo \otimes \lksl$-module as
 $$\V{p}  =  \bigoplus_{n=0} ^{\infty} \rho_n \otimes \LL{n}{p}. $$
\end{theorem}
This resolves the conjecture of \cite{C} stated at the end of Section 1.2 of that work.
There are then several small useful results that we establish about $\V{p}$,
\begin{enumerate}
\item
Corollary \ref{strong} tells us that
 $\V{p}$ is strongly generated by $ x= x(-1) {\bf 1} \otimes 1$, $x \in \{ e, f, h \} $ and the four vectors stated in \eqref{eq:genVp}.
 \item
Proposition \ref{simple} tells us that
 $\V{p}$ is a simple abelian intertwining algebra.
 \item
 Corollary \ref{cor:int} characterizes $\V{p}$ as the subalgebra of $\A{p}   \otimes \Pi(0)^{\frac{1}{2}}$ that is integrable with respect to the $\sltwo$-action of the horizontal subalgebra of $\lksl$,
$$ \V{p}  =  \left(  \A{p}   \otimes \Pi(0)^{\frac{1}{2}} \right)^{\rm{int}}.$$
\item Let us also note that $\V{p}$   has the following structure:
\begin{enumerate}
\item If $p \equiv 2 \ (\mbox{mod}\ 4)$,   $\V{p}$ is a $\frac{1}{2}\Z_{\ge 0}$-graded vertex operator superalgebra.
\item If $p \equiv  0 \ (\mbox{mod}\ 4)$,  $\V{p}$ is a  $ \Z_{\ge 0}$-graded vertex operator algebra.
\item If $p \equiv 1, 3\ (\mbox{mod}\ 4)$, $\V{p}$ is an abelian intertwining algebra.
\end{enumerate}
\end{enumerate}

\subsection{The $\R{p}$-algebra}
The vertex algebra $\R{p}$ appeared in \cite{A} is motivated by the free-field realization of the affine vertex algebra $L_{-3/2} (\mathfrak{sl}_2)$ that is isomorphic to $\mathcal R^{(2)}$.  These algebras are also studied in \cite{C}, where they are conjecturally identified as certain affine vertex algebras or vertex algebras for Argyres-Douglas theories.

The vertex algebra $\R{p}$ is defined as a subalgebra of $\V{p} \otimes F_{-\frac{p}{2}}$ where $F_{-\frac{p}{2}}$ is the abelain intertwining algebra associated to the weight lattice of $\sltwo$ rescaled by $\sqrt{-p}$. It is generated under operator products by $\lksl \otimes M(1)$ ($M(1)$ is the Heisenberg \voa) together with four vectors stated in \eqref{for-e12-p}. The $\R{p}$-algebra is related to the $\B{p}$-algebra of \cite{CRW}, which is characterized as
\[
 \B{p} = \left(\A{p} \otimes F_{-\frac{p}{2}} \right)^{U(1)}.
\]
The $\R{p}$-algebra is mainly studied in Section \ref{sec:Rp}. Most properties are inherited from $\V{p}$ and we list the main results as follows:
\begin{enumerate}
\item (Corollary \ref{r=rtilde})
\[
 \R{p} = \left(\V{p} \otimes F_{-\frac{p}{2}} \right)^{U(1)}
\]
and especially $\R{p}$ is simple.
\item As $\lksl \otimes  M(1)$-module
\[
{ \R{p}} \cong  \bigoplus_{\ell \in \Z}  \bigoplus_{s=0}^{\infty} \LL{\vert \ell  \vert + 2 s }{p} \otimes M_\varphi(1, -\ell)
\]
with $ M_\varphi(1, -\ell)$ certain Fock modules.
\item $\R{p} = \mbox{Ker} _{  \B{p} \otimes \Pi(0)^{\frac{1}{2}} } S $
\item
 (Corollaries  \ref{r=rtilde} and \ref{cor:int})
$
{\R{p}} =  \left(  \B{p} \otimes \Pi(0)^{\frac{1}{2}}\right)^{\rm{int}}$
 \end{enumerate}

\subsection{Tensor categories related to $\rm{KL}_k(\sltwo)$}

Let $V$ be a \voa{} and $\mathcal C$ a category of $V$-modules. A crucial problem is whether $\mathcal C$ has a (rigid) vertex tensor category structure. Having such a tensor category facilitates proving structural results as e.g. an effective theory of vertex operator superalgebra extensions \cite{HKL, CKM, CKM2}
and orbifolds \cite{M, M2}. As will be explained in the next subsection, we are able to employ our vertex tensor category findings together with just mentioned theory to prove powerful uniqueness results of \voa{} structures.

Let $\mathfrak{g}$ a simple Lie algebra, $k$ in $\C$ and $\rm{KL}_k(\mathfrak{g})$ be the category of ordinary modules for the simple affine \voa{} $L_k(\mathfrak{g})$ of $\mathfrak g$ at level $k$. A general aim is to establish rigid vertex tensor category structure on this category for all $\mathfrak g$ and $k$. Generically, that is if $k \notin \Q_{\geq -h^\vee}$, this has been achieved in seminal work by Kazhdan and Lusztig \cite{KL1}--\cite{KL5}.  For $k \in \Z_{\geq 1}$ this follows from \cite{H3, H4}, while for $k$ an admissible level the vertex tensor category structure has been proven to exist in \cite{CHY} and rigidity in the simply-laced case in \cite{C2}. In the accompanying work \cite{CY}, it is proven that semi-simplicity of
$\rm{KL}_k(\mathfrak{g})$ implies the existence of vertex tensor category. This result together with a main Theorem of \cite{M} and our Theorem \ref{thm;sl2decomp} implies that $\rm{KL}_k(\sltwo)$ for $k=-2+\frac{1}{p}$ and $p$ in $\Z_{\geq 1}$ is a rigid vertex tensor category and as such braided equivalent to a twist by some abelian $3$-cocycle of $\rm{Rep}(SU(2))$.
Together with \cite{KL1}--\cite{KL5}, \cite{H3, H4, CHY}, this result completes the case $\sltwo$ and thus
\begin{corollary}
For all $k \in \C$, the category of ordinary modules $\rm{KL}_k(\sltwo)$ is a rigid vertex tensor category.
\end{corollary}

Let $\rm{KL}_k(\sltwo)^{\rm{even}}$ be the full tensor subcategory whose simple objects are the $\LL{2n}{p}$. We also prove that $\rm{KL}_k(\sltwo)^{\rm{even}}\cong \rm{Rep}(SO(3))$ as symmetric tensor categories.

A corollary of the vertex tensor category structure is that we have many simple currents as discussed in Subsection \ref{sec:sc}. For example,
since $\R{p}$ is realized as a $U(1)$-orbifold of some larger abelian intertwining algebra one gets that the $\R{p}$-modules appearing in the decomposition are all simple currents due to results of \cite{M, CKLR}.

\subsection{Uniqueness of \voa{} structure}

Given two \voas{} $V$ and $W$ that share common properties, e.g. they have the same character or they have the same type of strong generators or
 they are isomorphic as modules for some common subalgebra. In such a case one usually would like to know if these two \voas{} are actually isomorphic leading to the general question: Under which assumptions can we guarantee that two \voas{} are isomorphic?
  For example a simple affine \voa{} is uniquely specified by the Lie algebra structure on its weight one subspace together with the invariant bilinear form restricted to the weight one subspace. Similarly \voas{} that are strongly generated by fields in weight one and $3/2$ are also uniquely specified by certain structure \cite{ACKL}.

We shall apply the correspondence between the \voa{} extensions and the commutative and associative algebra objects in the vertex tensor category \cite{HKL}.
We first use that the $U(1)$-orbifold of $\A{1}$ is nothing but the rank one Heisenberg \voa{} to deduce that a certain object in (a completion of) $\rm{Rep}(SO(3))$ can be given a unique simple commutative and associative algebra structure. Secondly this implies the uniqueness of corresponding extensions of $\lksl$, see Theorem \ref{thm:uniqueness}. Since simple current extensions also have the uniqueness property, we get the conclusion that a simple \voa{} $\mathcal X$ that is isomorphic as an $\lksl \otimes M(1)$-module to $\R{p}$ must even be isomorphic to $\R{p}$ as a \voa. A similar argument applies to $\B{p}$-algebras using the novel vertex tensor category results of the Virasoro algebra \cite{CFJRY}.
\begin{corollary}\rm{(Corollaries \ref{cor:uniquessRp} and \ref{cor:uniquessBp})} For $p$ in $\Z_{\geq 1}$ and $k=-2+\frac{1}{p}$, let $\mathcal X$ be a simple \voa such that $\mathcal X \cong \R{p}$
as an $\lksl \otimes M(1)$-module. Then $\mathcal X \cong \R{p}$ as \voas.

Analogously if a simple \voa{} $\mathcal Y\cong \B{p}$ as an $L^{\rm{Vir}}(c_{1, p},0) \otimes M(1)$-module, then $\mathcal Y\cong \B{p}$ as \voas.
\end{corollary}
This conclusion solves the conjectures of \cite{C} concerning $\cW$-algebras at boundary admissible levels and chiral algebras for Argyres-Douglas theories.
We now explain the $\cW$-algebra connections and turn to the physics at the end of this introduction.

\subsection{$\cW$-algebras and conformal embeddings}

Let $\mathfrak g$ be a simple Lie algebra, $f$ a nilpotent element in $\mathfrak g$ and $k$ a complex number. Then to this data one associates via quantum Hamiltonian reduction from the affine \voa{} $V^k(\mathfrak g)$ the universal $\cW$-algebra of $\mathfrak g$ at level $k$ corresponding to $f$, denoted by $\cW^k(\mathfrak g, f)$ \cite{KW3}. Let $\cW_k(\mathfrak g, f)$ denote the unique simple quotient of $\cW^k(\mathfrak g, f)$. The level $k$ is admissible if $k=-h^\vee + \frac{u}{v}$ and $u, v$ positive, coprime integers with $u\geq h^\vee$  if $v$ is coprime to the lacity of $\mathfrak g$ and $u\geq h$ otherwise. Here $h^\vee$ and $h$ denote the dual Coxeter and Coxeter number of $\mathfrak g$.
An interesting question that has been studied much recently is to classify $\cW^k(\mathfrak g, f)$ that are conformal extensions of affine \voas{} and moreover to understand their decomposition in terms of modules of this affine \voa{}. This has been particularly well understood if $f$ is trivial or minimal nilpotent \cite{AKMP, AKMP2, AKMP3}. In these cases one knows the operator product algebra and also has a powerful uniqueness theorem \cite{KW3, ACKL}.

Finding examples of conformal embeddings and branching rules for $\cW$-algebras corresponding to other nilpotent elements is difficult. We successfully give explicit $\cW$-algebras realizations of $\R{p}$ and $\B{p}$-algebras through character computations \cite{C, ACKL} together with our uniqueness results. Furthermore, the $\R{p}$-case is a conformal embedding.
 \begin{theorem}\label{thm:W} \rm{(Theorems \ref{thm:RpW} and \ref{thm:BpW})} For $p\in \mathbb Z_{\geq 2}$
 \begin{enumerate}
\item  Let $\ell = -\frac{p^2-1}{p}$ and $f$ a nilpotent element in $\mathfrak{sl}_{p+1}$ corresponding to the partition $(p-1, 1, 1)$ of $p+1$, then
 $\cW_\ell(\mathfrak{sl}_{p+1}, f) \cong \R{p}$ as \voas.
 \item Let $\ell = -\frac{(p-1)^2}{p}$, then
 $\cW_\ell(\mathfrak{sl}_{p-1}, f_{\text{sub}}) \cong \B{p}$ as \voas.
 \end{enumerate}
 \end{theorem}
The second statement solves the conjecture that $\B{p}$  is a simple quotient of an affine  $\cW$-algebra of type $A$ \cite{CRW}.

\subsection{Inverting Quantum Hamiltonian reduction}

The quantum Hamiltonian reduction realizes the Virasoro algebra at central charge $c_{1, p}$ as a certain cohomology of a complex associated to $\lksl$ with $k=-2+\frac{1}{p}$. The cohomology of  the $\LL{n}{p}$ are then corresponding Virasoro algebra modules and it is no problem to verify that $H_{\rm{DS}}^0(\V{p}) \cong \A{p}$ and $H_{\rm{DS}}^0(\R{p}) \cong \B{p}$ as Virasoro algebra modules. We work out the quantum Hamiltonian reduction of relaxed-highest weight modules (Proposition \ref{pom-1}) in order to prove that
\begin{theorem} \textup{(Theorem \ref{reduction})}
As \voas{}
\[
H_{\rm{DS}}^0(\R{p}) \cong \B{p},
\]
and as abelian intertwining algebras
\[
H_{\rm{DS}}^0(\V{p}) \cong \A{p}.
\]
\end{theorem}
The above Theorem resolves Conjecture 5.11 of \cite{C} and we will return to it when discussing the physics applications.
Having this Theorem in mind one sees that the statements
\[
{\R{p}} =  \left(  \B{p} \otimes \Pi(0)^{\frac{1}{2}}\right)^{\rm{int}}, \qquad
 \V{p}  =  \left(  \A{p}   \otimes \Pi(0)^{\frac{1}{2}} \right)^{\rm{int}}
\]
invert the quantum Hamiltonian reduction. This and also our other findings are interesting in the context of vertex algebras for $S$-duality and Argyres-Douglas theories as we will finally explain now.

\subsection{Vertex Algebras for $S$-duality}

Let $G$ be a compact Lie group with Lie algebra $\mathfrak g$ and let $\Psi$ be a complex number. One associates to this data so-called GL-twisted $\mathcal N=4$ superconformal four-dimensional gauge theories. $G$ is the gauge group and $\Psi$ the coupling of the theory and GL indicates the connection to the geometric Langlands program \cite{KaWi} (see also \cite{AFO, FG} for recent work on the connection between quantum geometric Langlands and $S$-duality). The physics motivation is to generalize Montonen-Olive electro-magnetic duality and the overall picture is that there are dualities between gauge theories associated to $G$ and either $G$ itself or its Langlands dual ${}^LG$ with coupling constants related by M\"obius transformation $\Psi \mapsto \frac{a\Psi +b}{c\Psi+d}$ with $\left( \begin{smallmatrix} a& b \\ c & d \end{smallmatrix}\right)$ in $\text{GL}(2, \mathbb Z)$. The inversion of $\Psi$ is referred to as $S$-duality. Vertex algebras appear at the intersections of three-dimensional topological boundary conditions, while categories of modules are attached to the various boundary conditions. The type of \voa{} depends on the type of boundary conditions, see \cite{CG, GR}. Most importantly, the conjecture is that the following object has the structure of a simple vertex operator superalgebra,
\[
A^n[\mathfrak g, \Psi] = \bigoplus_{\lambda \in P^+_n} V^k(\lambda) \otimes V^\ell(\lambda)
\]
 where $k, \ell$ are related to the coupling $\Psi$ via $\Psi=k+h^\vee$ and
 \[
 \frac{1}{k+h^\vee} + \frac{1}{\ell+h^\vee} = n \ \in  \ \mathbb Z_{\geq 1}.
 \]
Here $P^+_n$ is a subset of dominant integral weights depending on $n$ such that it includes all dominant integral weights that lie in the root lattice $Q$, i.e. $Q\cap P^+ \subset P^+_n$ and the $V^k(\lambda)$ are Weyl modules, namely the $V^k(\mathfrak g)$-modules induced from the irreducible highest-weight representation $\rho_\lambda$ of $\mathfrak g$ of highest-weight $\lambda$. Other interesting \voas{} appear by applying the quantum Hamiltonian reduction functor of either level $\ell$ or $k$ for some nilpotent element $f$ of $\mathfrak g$. The existence of these \voas{} is mostly open, except for $\mathfrak g$ simply-laced and $f$ principal nilpotent \cite{ACL} and $\mathfrak g =\mathfrak{sl}_2$ and $n=1, 2$ \cite{CG, CGL}. These latter cases are related to the exceptional Lie superalgebra $\mathfrak d(2, 1, \alpha)$ and its minimal $\cW$-superalgebra, the large $N=4$ superconformal algebra.
The algebras $A^n[\mathfrak g, \Psi]$ are expected to play an important role in quantum geometric Langlands, while its large $\Psi$-limit should relate to the classical geometric Langlands program. The expectation is that $A^n[\mathfrak g]$ is a deformable family of \voas{} in the sense of \cite{CL} so that the limit $\Psi\rightarrow \infty$ exists. Moreover, the simple quotient of this limit should be a \voa{} with $G$ as a subgroup of automorphisms and it should be of the form
\[
A^n[\mathfrak g, \infty] = \bigoplus_{\lambda \in P^+_n} \rho_\lambda \otimes V^\ell(\lambda)
\]
as $G\otimes V^\ell(\mathfrak g)$-module and $\ell= -h^\vee +\frac{1}{n}$. All these are conjectural, see Section 1.3.1 of \cite{CG} and our uniqueness result proves these large $\psi$-limit conjectures for $G=SU(2)$, i.e we have that \cite[Conjecture 1.2]{CG} is true for $G=SU(2)$,
\[
A^n[\mathfrak g, \infty] \cong \begin{cases}  \V{n} & \qquad n \ \text{even} \\ (\V{n})^{\mathbb Z_2} & \qquad n \ \text{odd}. \end{cases}
\]
Quantum Hamiltonian reduction of $\V{n}$ gives $\A{n}$, which is the large $\Psi$-limit of another such corner \voa.

\subsection{Chiral Algebras for Argyres-Douglas theories}

The second physics instance relevant to our work are chiral algebras of Argyres-Douglas theories. These are also four dimensional but $\mathcal N=2$ supersymmetric gauge theories \cite{AD} associated to pairs of Dynkin diagrams $(X, Y)$ of simple Lie algebras. Vertex algebras appear as chiral algebras of protected sectors of these gauge theories \cite{Beemetal} and in this instance the central charge, the rank of the Heisenberg subalgebra, affine subalgebras and their levels and the graded character of the chiral algebra can be determined from physics considerations, see e.g. \cite{BN1, CS}. The clear question is then if indeed a \voa{} with the desired properties exists and if it is uniquely determined by them. We also require that the chiral algebra is a simple \voa.

Set $X=A_1$ and $Y$ either of type $A$ or $D$. The Schur index and central charge of $(A_1,A_{2n})$ Argyres-Douglas theories coincide with the character
and the central charge of $L^{\text{Vir}}(c_{2, 2n+3},0)$ with $c_{2, n} = 1 - 6(2n + 1)^2/(4n + 6)$ \cite{Ras} and there is no flavor symmetry meaning that there is no Heisenberg or affine subalgebra.
In the case of $(A_1 , D_{2n+1})$, the physics data determine the chiral algebra as the simple affine vertex operator algebra of $\mathfrak{sl}_2$ at level $k = -4n/(2n + 1)$. The uniqueness of these \voas{} is obvious, i.e. any simple \voa{} whose character and central charge coincides with the simple Virasoro vertex operator algebra or simple affine vertex operator algebra $\lksl$ must be isomorphic to this \voa. The cases of the chiral algebras of Argyres-Douglas theories of types $(A_1, D_{2p})$ and $(A_1, A_{2p-3})$ are much more complicated. In \cite{C}, the Schur-index was identified with the one of the $\cW$-algbera of our Theorem \ref{thm:W} in the $(A_1, D_{2p})$-case and with the character of the $\B{p}$-algebra in the type $(A_1, A_{2p-3})$-case. Our uniqueness Theorems identify the chiral algebras of these Argyres-Douglas theories, see Section \ref{sec:AD}.
That is, we have for $p \geq 2$:
\begin{enumerate}
\item The chiral algebra of type $(A_1, D_{2p})$ is $\R{p}$.
\item  The chiral algebra of type $(A_1, A_{2p-3})$ is $\B{p}$.
\end{enumerate}

\subsection{Outlook}

Higher rank analogues of the triplet algebras are introduced by Feigin and Tipunin \cite{FT}. Not much is known about these algebras \cite{AdM-LCFT2, CM2} and they deserve further study. For example there are higher rank analogues of $\B{p}$ whose character coincides with a Schur index of a higher rank Argyres-Douglas theory \cite{C3,BN2}. Our current aim building on this work is to solve more decomposition problems of conformal embeddings of $\cW$-algebras and to understand quantum Hamiltonian reduction on the category of relaxed-highest weight modules and their spectrally flown images better.

\subsection{Organization of this work}

We start in section \ref{sec:2} by defining the $\V{p}$ and $\R{p}$-algebras along with stating a few basic properties.
The next two sections are then devoted to establish most of the structural results about $\V{p}$ and $\R{p}$ that we mentioned in the introduction.
The case $p=1$ is different (and much simpler) than the general case and is discussed in section \ref{sec:5}.
Our structural results are then used in section \ref{sec:6} to determine tensor category structure and to use this to prove the uniqueness of vertex operator algebra structure on our algebras. As a consequence we identify $\R{p}$ and $\B{p}$ with $\mathcal W$-algebras. The next section then uses these uniqueness results to identify $\R{p}$ and $\B{p}$ with chiral algebras of Argyres-Douglas theories.
In section \ref{sec:8} we study properties of the quantum Hamiltonian reduction functor from $\lksl$ to the Virasoro algebra. Especially we give a procedure that goes back from Virasoro modules to $\lksl$-modules. This is used to show that $\V{p}$ and $\R{p}$ are related to $\A{p}$ and $\B{p}$ via quantum Hamiltonian reduction.

\section{Introduction of the relevant algebras }\label{sec:2}

We recall \cite{A} and  \cite{A-2019} and realize the algebras that we are interested inside larger free field times lattice vertex algebras.
For this let $ p\in {\Z}_{\ge 2}$ and let $N^{(p)} $ be the following lattice
\begin{equation}
N^{(p)}= {\Z} \a + {\Z}\b + {\Z} \delta
\end{equation}
with the ${\Q}$--valued bilinear form $\la \cdot, \cdot \ra$ such
that
\begin{equation}
 \la \a , \a \ra = 1,  \quad \la \b , \b
\ra =-1 \quad \text{and} \quad \la \delta, \delta \ra = \frac{2}{p}
\end{equation} and all other products of basis vectors are zero.
Let $V_{ N^{(p)} } $ be the associated  abelian intertwining algebra
and set
$k = -2+ \frac{1}{p}$.
One defines the
 three elements
\begin{equation} \label{def-efh}
\begin{split}
e &= e^{\a + \b}, \\
 h &= -2 \b (-1) + \delta(-1),  \\
 f &= ( (k+1)  ( \a (-1) ^{ 2} - \a(-2) ) - \a(-1)
\delta(-1) +     (k+2)  \a(-1) \b (-1)  )   e^{-\a - \b}.
\end{split}
\end{equation}
Then the components of the fields $$Y(x,z) = \sum\limits_{n \in {\Z}} x(n)
z ^{-n-1}, \ x \in \{e,f,h\}$$ satisfy the commutations relations
for the affine Lie algebra $\asltwo$ of level $k$. Moreover, the subalgebra of $V_{ L ^{(p)} }$ generated by the
set $\{e,f,h\}$ is isomorphic to the simple vertex operator
algebra $\lksl$.

The screening operators that we need are
\bea
\label{scr-gen}   Q = \mbox{Res}_z Y( e^{\alpha + \beta - p  \delta}, z) , \quad  \widetilde{Q} = \mbox{Res}_z Y( e^{ - \frac{1}{p} (\alpha + \beta)  +\delta}, z).
\eea
They commute with the $\asltwo$--action. It is now useful to introduce some additional elements
following \cite{A-2019}.    Let
\begin{equation}\label{eq:gammamunu}
\begin{split}
 \gamma &:= \alpha + \beta - \tfrac{1}{k+2}\delta =  \alpha + \beta - p \delta,\\
  \mu &:= - \beta + \tfrac{1}{2} \delta, \\
  \nu &:= -\tfrac{k}{2} \alpha - \tfrac{k+2}{2}\beta  + \tfrac{1}{2}\delta =  \alpha - \tfrac{1}{2p}(\alpha+\beta)  + \tfrac{1}{2}\delta.
  \end{split}
\end{equation}
Then
\begin{equation}
 \langle \gamma, \gamma \rangle = \frac{2}{k+2} = 2p , \quad \langle \mu, \mu \rangle = - \la \nu, \nu \ra = \frac{k}{2},
 \end{equation}
and all other products are zero.  The screening charges then take the form
\bea
\label{scr-gen2}   Q = \mbox{Res}_z Y( e^{\gamma}, z) , \quad  \widetilde{Q} = \mbox{Res}_z Y( e^{ - \frac{\gamma}{p}}, z).
\eea
For our calculation, it is useful to notice that
\begin{equation}
\begin{split}
 \alpha &= \nu + \tfrac{k+2}{2}\gamma, \qquad
  \beta = -\tfrac{k+2}{2} \gamma + \tfrac{2}{k} \mu - \tfrac{k+2}{k} \nu,\\
 \delta &= - (k+2) \gamma + \tfrac{2 (k+2)}{k} \mu - \tfrac{2 (k+2)}{k} \nu.
 \end{split}
 \end{equation}
Let
\begin{equation}\label{eq:cd}
 c= \frac{2}{k} (\mu - \nu), \ d = \mu + \nu.
 \end{equation}
Then

 $$ \frac{p}{2} \delta =  -\frac{\gamma}{2} + \frac{c}{2} $$

Let $M$ be the subalgebra of $V_{N^{(p)}}$ generated by
$$ a  = e ^{\alpha + \beta}, \ a ^{*} =-\alpha(-1) e^{-\alpha -
\beta}. $$ Then $M$ is isomorphic to the Weyl vertex algebra. The Weyl vertex algebra is often also called the $\beta\gamma$-\voa{} and in physics the symplectic boson algebra.

Let $M_{\delta}(1)$ be the Heisenberg vertex algebra generated by the field $\delta (z) = \sum\limits_{n \in {\Z}} \delta(n) z ^{-n-1} $
and let $\Lambda_{\delta}^{(p)} := \frac{p\delta}{2} \Z$. Then
\begin{equation}
F_{\frac{p}{2}} := V_{\Lambda_\delta^{(p)}} = M_{\delta} (1) \otimes {\C}[ {\Z} \frac{p}{2} \delta]
\end{equation}
and
\begin{equation}
M \otimes F_{\frac{p}{2}}
\end{equation}
are subalgebras of $V_{N^{(p)}}$.

We have the following abelian intertwining algebra
\begin{equation}\label{defVp}
 \V{p} = \mbox{Ker}  _{ M \otimes F_{\frac{p}{2}} } \widetilde{Q}.
 \end{equation}
Moreover, $\lksl$ can be realized as a subalgebra of
 $M \otimes M_{\delta}(1) \subset  M \otimes F_{\frac{p}{2}}$, namely
\begin{equation}\label{def-ehf-cw}
\begin{split}
 e(z)  &= a (z), \\
 h(z) &= -2 : a ^{*} (z) a (z) : +  \delta(z) ,\\
 f(z)  &= - : a^{*} (z) ^{2} a (z) : +  k \partial_{z}
a^{*} (z) +    a^{*} (z) \delta(z).
\end{split}
\end{equation}
Since $ \widetilde{Q}$ commutes with the action of  $\asltwo$ we have that  $$\lksl \subset \V{p}. $$ Moreover, one can show that $Q$ acts as a derivation on ${\mathcal V} ^{(p)}$.
Note that if $p$ is even, then $\V{p}$ is a vertex superalgebra, while for odd $p$ it is not a vertex (super)algebra but only an abelian intertwining algebra.

Recall that the  screening operators are the zero-modes
$Q = e^{\gamma}_0$ and $\widetilde Q = e^{-\frac{\gamma}{p} } _0$.
We remark that
\begin{align*}
[Q, \widetilde{Q}] = (e^{\gamma}_0 e^{-\frac{\gamma}{p} })_0 = (\gamma(-1)e^{\frac{(p-1)\gamma}{p} })_0 = \frac{p}{p-1}(\partial e^{\frac{(p-1)\gamma}{p} })_0 = 0.
\end{align*}
There are thus four important vectors that are obviously in $\V{p}$, namely
\begin{equation}\label{eq:genVp}
\begin{split}
\tau_{(p)} ^+  &= e^{ \frac{p}{2} \delta},  \\
\overline{\tau}_{(p)} ^+  &= Q  e^{ \frac{p}{2} \delta},  \\
{\tau}_{(p)} ^- &= f(0) e^{ \frac{p}{2} \delta},   \\
 \overline{\tau}_{(p)} ^- &= -f(0)  Q e^{ \frac{p}{2} \delta}.
 \end{split}
 \end{equation}

Now we introduce the second algebra we are interested in, which we call the $\R{p}$-algebra.
Let $\varphi$ satisfy $\langle \varphi, \varphi\rangle = -\frac{2}{p}$ and let $\Lambda^{(p)}_\varphi = \frac{p\varphi}{2}\Z$ and $F_{-\frac{p}{2}}:= V_{\Lambda^{(p)}_\varphi}$.
The vertex algebra $\R{p}$ is defined to be
the subalgebra of $\V{p} \otimes F_{-\frac{p}{2}}$ generated by
$ x= x(-1) {\bf 1} \otimes 1$, $x \in \{ e, f, h \} $, $1 \otimes \varphi(-1) {\bf 1}$  and
\begin{equation}\label{for-e12-p}
\begin{split}
e_{\alpha_1,p} &:=  \frac{1}{\sqrt{2}}\  \tau^{+}_{(p)}  \otimes e^{ \frac{p}{2}\varphi} =   \frac{1}{\sqrt{2}}  e^{ \frac{p}{2}( \delta +  \varphi) },   \\
f_{\alpha_1,p }  &:= \frac{1}{\sqrt{2}} \  \overline{\tau} ^{-}_{(p)}  \otimes e^{- \frac{p}{2} \varphi} =   - \frac{1}{\sqrt{2}} f(0)   Q  e^{ \frac{p}{2}( \delta +  \varphi) },    \\
e_{\alpha_2,p }  &:= \frac{1}{\sqrt{2}} \ \overline{\tau} ^{+} _{(p)} \otimes e^{- \frac{p}{2} \varphi}  =  \frac{1}{\sqrt{2}} Q  e^{ \frac{p}{2}( \delta +  \varphi) },    \\
f_{\alpha_2,p }&:= \frac{1}{\sqrt{2}}\ {\tau} ^{-} _{(p) } \otimes e^{\frac{p}{2} \varphi}  =   \frac{1}{\sqrt{2}}  f(0)  e^{ \frac{p}{2}( \delta +  \varphi) }.
\end{split}
 \end{equation}
 The Heisenberg vertex algebra generated by $\varphi$ is denoted by $M_{\varphi} (1) $ and
  in general, $\R{p}$ is an extension of
$$ \lksl \otimes M_{\varphi} (1) $$ by the four fields of conformal weight $p/2$ in (\ref{for-e12-p}). Set $M_{\delta, \varphi}(1)=M_\delta(1)\otimes M_{\varphi}(1)$
and let
\begin{equation}
\PiO = M_{\varphi, \delta} (1) \otimes {\C}[ \Z \frac{p}{2} ( \delta + \varphi) ] \subset F_{\frac{p}{2}} \otimes F_{-\frac{p}{2}}.
\end{equation}
Then $\PiO$ contains a rank one isotropic lattice \voa.
In general we have that
 $$ \mathcal R^{(p)} \subset  (M \otimes \PiO)^{\rm{int}}, $$
where $(M \otimes \PiO)^{\rm{int}}$ is the maximal $\sltwo$--integrable submodule of $M \otimes \PiO$.
The cases $p=2, 3$ have been studied:
\begin{theorem}   \cite{A, AKMP}
\begin{enumerate}
\item   $ \R{2}  \cong   L_{-3/2} (\mathfrak{sl}_3)$.
\item  $\R{3}  \cong  W_{-8/3} (\mathfrak{sl}_4, f_{\theta})$ with $f_\theta$ minimal nilpotent.
\item For $p=2,3$ we have  $\R{p} = \mbox{Ker} _{M \otimes \PiO } \widetilde Q = (M \otimes \PiO)^{\rm{int}}.$
\end{enumerate}
\end{theorem}

\section{The $\V{p}$-algebra}

We first study $\V{p}$ and $\R{p}$ will inherit many properties from $\V{p}$.

 \subsection{ From the doublet $\A{p}$ to $\V{p}$ }

We will now realize $\R{p}$ and $\V{p}$ by lifting certain well-known extensions called doublet algebras $\A{p}$ of the Virasoro \voa, following  \cite{A-2019}.
The various useful lattice vectors have been recorded in \eqref{eq:gammamunu}--\eqref{eq:cd}.
For $\underline{c} \in \mathbb{C}$, we denote the universal Virasoro \voa{} of central charge $\underline{c}$ by $V^{\rm{Vir}} (\underline{c}, 0)$.
For any two co-prime positive integers $p, q$ we set
\[
c_{q, p} := 1 - 6\frac{(p-q)^2}{pq}.
\]
Let $\omega$ be  the conformal vector in $V^{Vir} (c_{1,p}, 0)$.
Define
\begin{equation}
 \Pi(0) := M_{c,d} (1) \otimes {\C}[{\Z} c]\qquad \text{and} \qquad  \Pi(0)^{\frac{1}{2}} = M_{c,d} (1) \otimes {\C}[{\Z} \frac{c}{2}].
\end{equation}
Since $\langle c, c\rangle =0$ these are \voas{} that contain a rank one isotropic lattice \voa{} as their subalgebra.
There is an injective homomorphism of vertex algebras
$$\Phi : V^{k} (\sltwo) \rightarrow V^{\rm{Vir}} (c_{1,p}, 0) \otimes \Pi(0)  $$
such that
 \begin{equation}
 \begin{split}
e & \mapsto  e^{c  }, \label{def-e-3} \\
h & \mapsto  2 \mu(-1), \\
f & \mapsto  \left[     (k+2) \omega   -\nu(-1)^{2} -  (k+1) \nu(-2) \right] e^{-c }.
\end{split}
 \end{equation}
The famous doublet $\A{p}$ and triplet $\W{p}$ algebras are realized as \cite{AdM-doublet}:
\begin{equation}
\begin{split}
\A{p} &=   \mbox{Ker}_{ V_{{\Z} \frac{\gamma}{2}}    }\widetilde Q, \\
\W{p} &=  \mbox{Ker}_{ V_{{\Z} \gamma}    }\widetilde Q.
\end{split}
\end{equation}
The triplet $\W{p}$ are \voas, while the doublets are vertex superalgebras for even $p$ and otherwise only abelian intertwining algebras.
 Recall that
$\A{p}$ is generated by the doublet $a^\pm$ together with the Virasoro element $\omega$. These are explicitely
$$ a^+ = Q a^-, \qquad a^- =e^{-\frac{\gamma}{2}}\quad \text{and} \quad \omega = \frac{1}{4 p} \gamma(-1) ^2 + \frac{p-1} {2p} \gamma (-2). $$

The abelian intertwining algebras $\A{p}$ and $\Pi(0)^{\frac{1}{2}}$ both have a $\Z_2$-action with invariant subalgebras $\W{p}$ and $\Pi(0)$. So that they decompose as $\W{p}$ and $\Pi(0)$-modules
\[
\A{p} = \A{p}_0 \oplus \A{p}_1 \qquad  \text{and} \qquad
\Pi(0)^{\frac{1}{2}} = \Pi(0)^{\frac{1}{2}}_0 \oplus \Pi(0)^{\frac{1}{2}}_1,
\]
with $\A{p}_0=\W{p}$ and $\Pi(0)^{\frac{1}{2}}_0=\Pi(0)$. The diagonal $\Z_2$-orbifold is thus
\begin{equation}\label{eq:ApPiZ2-dec}
\left(\A{p} \otimes \Pi(0)^{\frac{1}{2}}\right)^{\Z_2} \cong \W{p} \otimes \Pi(0) \oplus  \A{p}_1 \otimes \Pi(0)^{\frac{1}{2}}_1.
\end{equation}
We have that $\V{p} \subset  \left(\A{p} \otimes \Pi(0)^{\frac{1}{2}}\right)^{\Z_2} $ and the expressions for the generators are
\begin{equation}
\begin{split}
 \tau_{(p)} ^+  &= a^{-} e^{\frac{c}{2}  },  \\
  \overline{\tau}_{(p)} ^+  &=a^{+} e^{\frac{c}{2}  },   \\
     {\tau}_{(p)} ^- &= f(0) a^{-} e^{\frac{c}{2} }    \\
     \overline{\tau}_{(p)} ^- &= -f(0)  a^{+} e^{\frac{c}{2} }.
 \end{split}
\end{equation}
Note that $\mathcal V^{(p)}$ can also be realized as a subalgebra of $M \otimes F_{\frac{p}{2}}$, we need to consider
screening operator
\begin{equation}\label{operatorS}
S =  e^{\alpha}_0 = \mbox{Res}_z Y(e ^{\alpha}, z),
\end{equation}
since $M= \mbox{Ker}_{\Pi(0)}S$ (\cite{FMS}).

\begin{remark}
Note that the screening operator $S$ can be obtained as
$$ S= \mbox{Res}_ z Y ( e^{\gamma / 2p + \nu}, z) =   \mbox{Res}_ z Y ( v_{2,1} \otimes e^{\nu} , z), $$
where $v_{2,1} = e^{\gamma / 2p}$ is a singular vector for the Virasoro algebra with conformal weight $\frac{3}{4} k + 1$.
This screening operator has also appeared in \cite{A-2019}.
\end{remark}

\begin{proposition}\label{prop:VpviaAp}
$ \V{p} = \mbox{\rm{Ker}}_{\left(\A{p} \otimes \Pi(0)^{\frac{1}{2}}\right)^{\Z_2} } S$.
\end{proposition}
\begin{proof}
Combining $\A{p} =   \mbox{Ker}_{ V_{{\Z} \gamma/2}}\widetilde Q$, $M= \mbox{Ker}_{\Pi(0)}S$ and $\V{p}  =  \mbox{Ker}_{M \otimes F_{\frac{p}{2}} } \widetilde Q$ together with computing that $\Pi (0)  \otimes F_{\frac{p}{2}} $ is the diagonal $\Z_2$-orbifold of $V_{{\Z} \gamma / 2} \otimes \Pi(0)^{\frac{1}{2}}$
we immediately get the claim
\bea
\V{p}  & = & \mbox{Ker}_{M \otimes F_{\frac{p}{2}} } \widetilde Q \nonumber \\
&=& \mbox{Ker}_{  \Pi (0)  \otimes F_{\frac{p}{2}}  }\widetilde Q\   \bigcap\    \mbox{Ker}_{  \Pi  (0)  \otimes F_{\frac{p}{2}}  }  S  \nonumber \\
&=&   \mbox{Ker}_{ \left(V_{{\Z} \gamma / 2} \otimes \Pi(0)^{\frac{1}{2}}\right)^{\Z_2}   }\widetilde Q  \ \bigcap  \  \mbox{Ker}_{ \left(V_{{\Z} \gamma / 2} \otimes \Pi(0)^{\frac{1}{2}}\right)^{\Z_2}     }  S  \nonumber \\
& =&\mbox{Ker}_{\left(\A{p} \otimes \Pi(0)^{\frac{1}{2}}\right)^{\Z_2} } S. \nonumber
\eea
\end{proof}

Let $L^{\rm{Vir}} (c,  h)$ denote the irreducible lowest-weight module of the Virasoro algebra at central charge $c$ of lowest-weight $h$.
Now we take the following decomposition of $\A{p}$ which follows from e.g. \cite[Thm 1.1 and 1.2]{AdM-triplet}
\begin{equation}\label{eq:decAp}
 \A{p} = \bigoplus_{n = 0} ^{\infty} \rho_n \otimes L^{\rm{Vir}} (c_{1,p},  h_{1,n+1} ),
 \end{equation}
where $\rho_n$ is $n+1$--dimensional irreducible representation of $\sltwo$ and
$$h_{1,n+1} = \frac{ ( 1 - (n+1) p) ^2 - (p-1)^2}{4 p}. $$
Note that the additional $\sltwo$-action is defined by $e = Q$, $h = \frac{\gamma(0)}{p}$ and certain action of $f$. Let
$$ v_{1, n, j} :=  Q ^ j e^{ - \frac{n-1}{2} \gamma} \in \A{p},\quad j =0, \ldots, n-1.$$
Then $v_{1, n, j}$ is the  highest weight vector in  $L^{\rm{Vir}} (c_{1,p},  h_{1,n} )$ and $\bigoplus\limits_{j=0}^{n-1} \C v_{1, n, j}$ is isomorphic to $\rho_{n-1}$ with respect to the additional $\sltwo$-action, in which $v_{1, n, j}$ is the weight vector of the weight $2j-n+1$.
We denote by $\LL{s}{p}$ the irreducible highest-weight representation of $\asltwo$ whose top level is $\rho_s$ and on which the central element acts by multiplication by the level $k=-2+\frac{1}{p}$. Common notations are
 $$\LL{s}{p} = L_{A_1} ( (k +2 -s) \Lambda_0 + s \Lambda_1) = L_k(s\omega_1) \quad (s \in {\Z}_{\ge 0}), $$
 with $\Lambda_0, \Lambda_1$ the affine fundamental weights and $\omega_1$ the fundamental weight of $\sltwo$.
\begin{proposition} \cite[Proposition 6.1]{A-2019}. \label{prop:Ls-in-LVir}
 For every $s \in {\Z}_{\ge 0}$ and $j = 0, \ldots, s$, we have
 $$ \widetilde{\phi}_{s, j} \colon \LL{s}{p} \xrightarrow{\sim} \lksl. (v_{1, s+1, j} \otimes e^{\frac{s}{2} c}) \subset  L^{\rm{Vir}} (c_{1,p},  h_{1,s+1} ) \otimes \Pi (0)^{\frac{1}{2}}. $$
\end{proposition} We present an important Lemma on uniqueness of singular vector in  $ L^{\rm{Vir}} (c_{1,p},  h_{1,s} ) \otimes \Pi (0)^{\frac{1}{2}}$.  A more general version will be studied in \cite{ACGY-inverse}.

\begin{lemma}  \label{ker} Assume that $w$ is a singular vector for $\asltwo$ in $ L^{\rm{Vir}} (c_{1,p},  h_{1,s} ) \otimes \Pi (0)^{\frac{1}{2}}$ with dominant integral weight.
Then, $ w = v_{1, s, j} \otimes e^{\frac{s-1}{2} c}$.
\end{lemma}
\begin{proof}

Let $\widehat{\mathfrak b}$ be the Borel subalgebra of $\asltwo$ generated by $e(n) = e^c _n$ and $h(n) = 2 \mu(n)$.
We  consider $ L^{\rm{Vir}} (c_{1,p},  h_{1,s} ) \otimes \Pi (0)^{\frac{1}{2}}$ as a module for $\widehat{\mathfrak b}$.

Assume that $W \subset  L^{\rm{Vir}} (c_{1,p},  h_{1,s} ) \otimes \Pi (0)^{\frac{1}{2}}$ is any  non-zero  $\lksl)$--submodule which is integrable with respect to $\sltwo$. Using action of  $\widehat{\mathfrak b}$, we see that $W$ contains vector $ u_1 \otimes  e^{\frac{m-1}{2} c}$ for certain  $u_1 \in L^{\rm{Vir}} (c_{1,p},  h_{1,s} ) $ and $m \in {\Z}$. By using the action of the Sugawara Virasoro element  $L_{\rm{Sug}}(n)$ we easily get that
$$v_{1, s, j} \otimes  e^{\frac{m-1}{2} c} \in W. $$ Direct calculation then shows that $v_{1, s, j} \otimes  e^{\frac{m-1}{2} c}$ is singular if and only if $m= s$.
\end{proof}

\begin{corollary}
$ \mbox{\rm{Ker}} _{   L^{\rm{Vir}} (c_{1,p}, h_{1, n+1} )   \otimes \Pi(0)^{\frac{1}{2}} } S \cong  \LL{n}{p}. $
\end{corollary}
\begin{proof}Since $v_{1,n+1,0} \otimes e^{\frac{n}{2} c} = e^ { -  \frac{n}{2} (\gamma -c)} $ and $S = e^{ \gamma / 2p + \nu} _0$, the relation
$ (  \gamma / 2p + \nu ,  \gamma - c) = 0$
implies that
$ S ( v_{1,n+1,0} \otimes e^{\frac{n}{2} c} ) = 0.$
Since
$ [S, Q] = 0$, we have that
$S   ( v_{1,n+1, j } \otimes e^{\frac{n}{2} c} ) = Q^j S ( v_{1,n+1,0} \otimes e^{\frac{n}{2} c} ) =0. $
\end{proof}
Recall that $\A{p} = \A{p}_0 \oplus \A{p}_1$ and $\Pi(0)^{\frac{1}{2}} = \Pi(0)^{\frac{1}{2}}_0 \oplus \Pi(0)^{\frac{1}{2}}_1$,
so that
\begin{equation}
\A{p} \otimes \Pi(0)^{\frac{1}{2}} = \left(\A{p} \otimes \Pi(0)^{\frac{1}{2}}\right)^{\Z_2} \oplus
\A{p}_1 \otimes \Pi(0) \oplus  \W{p} \otimes \Pi(0)^{\frac{1}{2}}_1.
\end{equation}
We inspect that
\[
\mbox{Ker}_{\A{p}_1 \otimes \Pi(0)}S =  \mbox{Ker}_{ \W{p} \otimes \Pi(0)^{\frac{1}{2}}_1}S = 0,
\]
so that Proposition \ref{prop:VpviaAp} improves to
\begin{corollary}\label{cor;VpviaAp}
$ \V{p} = \mbox{\rm{Ker}}_{\mathcal A (p) \otimes \Pi(0)^{\frac{1}{2}} } S$.
\end{corollary}

 \subsection{ The  $\sltwo$-action }

 \begin{theorem} \label{thm:Vp-dec}
$\V{p}$ decomposes as a $\sltwo \otimes \lksl$--module as
 $$\V{p}  =  \bigoplus_{n=0} ^{\infty} \rho_n \otimes \LL{n}{p}. $$
 \end{theorem}
 \begin{proof}
The decomposition of $\A{p}$ \eqref{eq:decAp} yields
 \bea
\A{p} \otimes \Pi(0)^{\frac{1}{2}}  &=& \bigoplus_{n=0} ^{\infty}  \rho_n \otimes  \left(L^{\rm{Vir}} (c_{1,p}, h_{1, n+1})  \otimes \Pi(0)^{\frac{1}{2}}  \right), \nonumber
 \eea
so that
$$ \V{p}  =  \mbox{Ker} _{   \A{p}   \otimes \Pi(0)^{\frac{1}{2}} } S \cong   \bigoplus_{n=0} ^{\infty} \rho_n \otimes \LL{n}{p}. $$
 \end{proof}

Strictly speaking the action of $\sltwo$ is obtained in \cite{ALM} for triplet. But it extends easily for doublet:
\begin{remark} In \cite{ALM} it was proven that the Lie algebra $\sltwo$ acts on $\W{p}$ by derivations.  Let us recall the main steps in the proof.
\begin{itemize}
\item We use decomposition of $\W{p}$ as a $\text{Vir} \otimes \sltwo$--module from \cite{AdM-triplet}.
\item We use the fact that screening operator $Q$ is a derivation and  we put  $e=Q$.
\item We construct an automorphism of order two $\Psi$ of singlet algebra which extends to an autumprhism of $\mathcal W^{(p)}$.
\item The action of operator $f$ is given by $f= - \Psi^{-1} Q \Psi$.
\end{itemize}

But the same proof  with minor modifications implies that  $\sltwo$ acts on   $\mathcal A^{(p)}$  by derivations. So:

\begin{itemize}
\item Using \cite{AdM-doublet}, the  doublet algebra $\A{p}$ is a $\text{Vir} \otimes \sltwo$--module and it decomposes as
$$  \mathcal A^{(p)} =  \bigoplus_{n=0} ^{\infty}  \rho_n \otimes  L^{\rm{Vir}} (c_{1,p}, h_{1, n+1}) . $$
\item The screening operator $Q$ is a derivation  of  $  \mathcal A^{(p)}$  and  we can  put  $e=Q$.
\item The automorphism $\Psi$ of   $\W{p}$, easily extends to an automorphism of order two of $ \mathcal A^{(p)}$ such that
$ \Psi( a^+ ) = a^-. $
\item  The operator $f$ can be again defined as $f= - \Psi^{-1} Q \Psi$.
\end{itemize}

\end{remark}

\begin{theorem} \cite{ALM}
The Lie algebra $\sltwo$ acts on $\A{p}$ by derivations and
$$ \left( \mathcal A^{(p)} \right)^{\sltwo} = L^{\rm{Vir}} (c_{1,p}, 0). $$
\end{theorem}

This immediately gives an $\sltwo$-action of $ \A{p} \otimes \Pi(0)^{\frac{1}{2}}$ and from the above Theorem we see that this action restricts to an action on the $\V{p}$-subalgebra, and so
\begin{corollary}\label{sl2der}
The Lie algebra $\sltwo$ acts on $\V{p}$ as derivations.
\end{corollary}

 Essentially the same argument as the one of Proposition 1.3 of \cite{AdM-triplet} gives now

 \begin{corollary}\label{strong}
 $\V{p}$ is strongly generated by $ x= x(-1) {\bf 1} \otimes 1$, $x \in \{ e, f, h \} $ and the four vectors stated in \eqref{eq:genVp}.
 \end{corollary}

Note that $\mathcal V^{(p)}$   has the following structure:

\begin{itemize}
\item If $p \equiv 2 \ (\mbox{mod} \ 4)$,   $\V{p}$ is a $\frac{1}{2}\Z_{\ge 0}$--graded vertex operator superalgebra.
\item If $p \equiv  0 \ (\mbox{mod} \ 4)$,  $\V{p}$ is a  $ \Z_{\ge 0}$--graded vertex operator algebra.
\item If $p \equiv 1, 3 \ (\mbox{mod} \ 4)$, $\V{p}$ is an abelian intertwining algebra.
\end{itemize}

We also have the following vertex subalgebra:
$$( \V{p} )^{\Z_2} = \mbox{ Ker} _{ \W{p} \otimes \Pi(0)^{\frac{1}{2}} } S  \cong   \bigoplus_{n=0} ^{\infty} \rho_{2 n} \otimes \LL{ 2 n}{p}. $$
In all cases,  Since $ ( \V{p} )^{\Z_2} $ is a subalgebra of the \voa{}  $\W{p} \otimes \Pi(0)^{\frac{1}{2}}$, we have that  $  (\V{p} )^{\Z_2}$  has the structure of a  vertex operator algebra.

Moreover, since $\mbox{Aut} (\W{p}) = PSL(2, \C)$ (cf. \cite{ALM}), we get:

\begin{corollary} The group $G= PSL(2, \C)$ acts on $( \V{p} )^{\Z_2}$ as its automorphism group, and we have the following decomposition of  $( \V{p} )^{\Z_2}$ as $G \otimes \lksl$--module:
$$( \V{p} )^{\Z_2} = \bigoplus_{n=0} ^{\infty} \rho_{2 n} \otimes \LL{2n}{p}. $$
\end{corollary}

 \subsection{Simplicity of $\V{p}$}

 \begin{proposition}\label{simple}
 $\V{p}$ is a simple abelian intertwining algebra.
 \end{proposition}
 \begin{proof}
 Assume that $ 0 \ne I  \subseteq \V{p}$  is a non-trivial ideal in $\V{p}$.  Then $I$ is a  $\lksl$--module in $\rm{KL}_{k}$ for $k=-2 +1/p$, and therefore it must contain a non-trivial singular vector with dominant weight. Using  Lemma   \ref{ker}  we get that
 $$ v_{1, n_0+1, j} \otimes e^{\frac{n_0}{2} c } \in I$$
 for certain $n_0 \in {\Z}_{\ge 0}$, $ 0 \le j \le n_0$.  We can take    $n_0$ to be minimal   with this property.  Assume that $n_0 >0$. Since $\A{p}$ is simple and generated by $a^{\pm}$,
  there is $m \in \frac{1}{2} {\Z}$, such that
 $$ a^+ _m  v_{1, n_0 +1, j}   =  C     v_{1, n_0 , j' }  \quad  \mbox{or}  \quad a^-_m  v_{1, n_0 +1, j}   =  C   v_{1, n_0 , j' } $$
 for certain $C \ne 0$ and $ 0 \le j' \le  n_0 -1$.
 Now applying the action of the four generators  \eqref{eq:genVp} of $\V{p}$ on    $v_{1, n_0+1, j} \otimes e^{\frac{n}{2} c } $ we get
 $$v_{1, n_0, j' } \otimes e^{\frac{n_0-1}{2} c  }  \in I. $$
 This  is in contradiction with minimality of  $n_0$. Therefore $n_0 = 0$, and ${\bf 1 } \in I$. So $I= \V{p}$. This proves the simplicity of  $ \V{p}$
 \end{proof}

 \section{ The  $\R{p}$ vertex algebra}\label{sec:Rp}

 The three abelian intertwining algebras $\V{p}, \A{p}$ and $F_{-\frac{p}{2}}$
 have  a natural $U(1)$-action. In the first two cases it is just obtained by exponentiating the action of the Cartan subalgebra of $\sltwo$, while in the last one it is obtained by exponentiating $\varphi(0)$. This action gives the decompositions in terms of $(\V{p})^{U(1)}\otimes U(1), \A{p})^{U(1)}\otimes U(1)$ and
 $M_\varphi(1) \otimes U(1)$ modules
 \begin{equation}
 \begin{split}
 \V{p} &\cong  \bigoplus_{\ell\in \Z} \V{p}_\ell \otimes \C_\ell, \\
 \A{p} &\cong  \bigoplus_{\ell\in \Z} \A{p}_\ell \otimes \C_\ell, \\
 F_{-\frac{p}{2}} &\cong \bigoplus_{\ell\in \Z} M_\varphi(1, \ell) \otimes \C_\ell,
 \end{split}
 \end{equation}
 with
  \begin{equation}
 \begin{split}
&\V{p}_\ell  \cong \bigoplus_{s=0}^{\infty} \LL{\vert \ell  \vert + 2 s }{p} \\
&\A{p}_\ell \cong \bigoplus_{s=0}^{\infty} L^{\rm{Vir}} (c_{1,p}, h_{1, \vert \ell \vert+2s+1})  \\
& M_\varphi(1, \ell) \cong M_{\varphi}(1).e^{\ell\frac{p\varphi}{2}}
 \end{split}
 \end{equation}
 The $\B{p}$-algebra of \cite{CRW} is the diagonal $U(1)$-orbifold of $\A{p} \otimes F_{-\frac{p}{2}}$
 \begin{equation}
 \begin{split}
 \B{p} &= \left(\A{p} \otimes F_{-\frac{p}{2}} \right)^{U(1)}\\
&= \bigoplus_{\ell \in \Z}  \A{p}_\ell \otimes M_\varphi(1, -\ell) \\
&= \bigoplus_{\ell \in \Z}  \bigoplus_{s=0}^{\infty} L^{\rm{Vir}} (c_{1,p}, h_{1, \vert \ell \vert+2s+1})  \otimes M_\varphi(1, -\ell).
 \end{split}
 \end{equation}
 Similarly the
 diagonal $U(1)$-orbifold of $\V{p} \otimes F_{-\frac{p}{2}}$ is
 \begin{equation}
 \begin{split}
 \widetilde{ \R{p}} :&= \left(\V{p} \otimes F_{-\frac{p}{2}} \right)^{U(1)}\\
&= \bigoplus_{\ell \in \Z}  \V{p}_\ell \otimes M_\varphi(1, -\ell) \\
&= \bigoplus_{\ell \in \Z}  \bigoplus_{s=0}^{\infty} \LL{\vert \ell  \vert + 2 s }{p} \otimes M_\varphi(1, -\ell).
 \end{split}
 \end{equation}
The vectors \eqref{for-e12-p} are in $\widetilde{ \R{p}}$ and hence $\R{p}\subset \widetilde{ \R{p}}$. We will prove in a moment that they actually coincide.

Recall \cite{CRW}  that the $\B{p}$ algebra  is generated by
$ a^+ \otimes e^{\varphi/ 2}, \ a^{-} \otimes e^{-\varphi / 2}, \varphi(-1)$  and  $\omega$.
The $U(1)$-invariant part of the isomorphism of Corollary \ref{cor;VpviaAp} says that
$$ \widetilde{\R{p}} = \mbox{Ker} _{  \B{p} \otimes \Pi(0)^{\frac{1}{2}} } S. $$

\subsection{Identification of the quotient}

As before, we assume that $k=-2+\frac{1}{p}$ for $p$ a positive integer and we denote by $\LL{s}{p}$ the simple highest-weight module of $\lksl$ of highest-weight $s\omega_1$.
Let $s$ be a non-negative integer, then from Lemma \ref{ker} we have that
\begin{equation}
\begin{split}
 \left( L^{\rm{Vir}} (c_k, h_{1,s+1} ) \otimes  \Pi(0)\right)^{\rm{int}} &\cong \begin{cases}  \LL{s}{p} & \qquad \text{if} \ s \ \text{is even} \\ 0 & \qquad \text{if} \ s \ \text{is odd} \end{cases}\\
 \left(L^{\rm{Vir}} (c_k, h_{1,s+1} ) \otimes  \Pi(0)_1^{\frac{1}{2}}\right)^{\rm{int}} &\cong \begin{cases}     0  & \qquad \text{if} \ s \ \text{is even} \\ \LL{s}{p}  & \qquad \text{if} \ s \ \text{is odd} \end{cases}.
\end{split}
\end{equation}
One can also define twisted modules, e.g. let $w_s$ be a twisted highest-weight vector satisfying
$$ e( m-1) w_s = 0,\qquad  f(m+2) w_s = 0, \qquad  h(m)  w_s =  \delta_{m,0} (s-2)  w_s \quad \text{for} \ m \ge 0.$$
  Then $w_s$ generates a module that we denote by
  $ \rho_1 (\LL{s-\frac{1}{p}}{p} ).$

  \begin{lemma} \label{quotient} Let $k=-2+\frac{1}{p}$ with $p$ a positive integer and $s$ a non-negative integer, then
  \begin{enumerate}
  \item For $s$ even
  \[
  0 \rightarrow \LL{s}{p} \rightarrow L^{\rm{Vir}} (c_k, h_{1,s+1} ) \otimes  \Pi(0) \rightarrow  \rho_1 \left(\LL{s-\frac{1}{p}}{p} \right) \rightarrow 0
  \]
  \item For $s$ odd
  \[
  0 \rightarrow \LL{s}{p} \rightarrow L^{\rm{Vir}} (c_k, h_{1,s+1} ) \otimes  \Pi(0)_1^{\frac{1}{2}} \rightarrow  \rho_1 \left(\LL{s-\frac{1}{p}}{p} \right) \rightarrow 0
  \]
  \end{enumerate}
  \end{lemma}

  \begin{proof}
Consider the vector
  $$  w_s = v_{1, s+ 1} \otimes e^{ \frac{s}{2} c - c}. $$
  Let us prove that $w_s$ is cyclic.
    Let $ W= \lksl. w_s$.  Assume that $s $ is even (the case $s$ is odd is completely analogous).  By applying operators $e(-1)$ and $f(1)$ we get
  $$  v_{1, s+ 1} \otimes e^{ \frac{s}{2} c + m  c} \in W \quad (\forall m \in {\Z}). $$
  Using  the action of $\widehat{ \mathfrak b}$, we get that
  $$ v_{1, s+1} \otimes z  \in W  \quad (\forall z  \in  \Pi(0)).  $$
   Then applying action of $L_{sug}(n)$ we obtain
   $$L^{\rm{Vir}} (c_k, h_{1,s+1} ) \otimes  \Pi(0) = W.$$
 Next we notice that
  $e(-1) w_s =  v_{1, s+ 1} \otimes e^{ \frac{s}{2} c } $ which is singular in   $L^{\rm{Vir}} (c_k, h_{1,s+1} ) \otimes  \Pi(0)^\frac{1}{2}$.
  So in the quotient module $\overline W_s$,   the vector
  $$ \overline w_s = w_s +   \mathcal L_s ^{(p)} $$ satisfies twisted highest weight condition:
  $$ e( m-1) \overline w_s = 0,\qquad  f(m+2) \overline w_s = 0, \qquad  h(m)  \overline w_s =  \delta_{m,0} (s-2)  \overline w_s \quad \text{for} \ m \ge 0.$$
    Applying again the action of  $\widehat{ \mathfrak b}$ we see that $ \overline w_s$ is unique singular vector in $\overline W_s$.
  We conclude     that the quotient module is isomorphic to the module $ \rho_1 \left(\LL{s-\frac{1}{p}}{p} \right)$.
  \end{proof}
  As a consequence, we get another nice characterization of our algebras:
\begin{corollary}\label{cor:int}
The algebras satisfy
$$ \V{p}  =  \left(  \A{p}   \otimes \Pi(0)^{\frac{1}{2}} \right)^{\rm{int}} \qquad \text{and} \qquad
 \widetilde{\R{p}} =  \left(  \B{p} \otimes \Pi(0)^{\frac{1}{2}}\right)^{\rm{int}}.$$
\end{corollary}

\section{The case $p=1$}\label{sec:5}

The case $p=1$ can be realized very similarly to the case $p\geq 1$. The only difference is that one doesn't need the screening operators $Q$ and $\widetilde{Q}$.
Consider the algebras $\A{1}= F_{\frac{1}{2}}$ and $\V{1}$, $\A{1}$ is an abelian intertwining algebra that is a simple current extension of $L_1(\sltwo)$:
$$ \A{1} =  L_1(\sltwo) \oplus   \LL{1}{1}. $$
Thus $\A{1}$ has the natural $\sltwo \times$ Vir action so that
$$ \A{1}  = \bigoplus_{n=0} ^{\infty}  \rho_n \otimes L ^{\rm{Vir}} \left(1, \frac{n^2}{4}\right). $$

By  Lemmas \ref{ker} and \ref{quotient} we have the following realization of irreducible $\asltwo$--module  $\LL{n}{1} =L_{-1} (n \omega_1) $, which has the  highest weight $-(1+n) \Lambda_0 + n\Lambda_1$:
$$ \LL{n}{1} = L_{-1} (n \omega_1)   = \left( L ^{\rm{Vir}} (1, \frac{n^2}{4}) \otimes \Pi(0)^{\frac{1}{2}} \right) ^{\rm{int}} . $$

Define $$\V{1} =\left(  \A{1} \otimes \Pi(0)^{\frac{1}{2}}  \right) ^{\rm{int} } $$
and apply the same arguments as for $\V{p}$ to give
\begin{proposition} $\V{1}$ is an abelian intertwining algebra, $\sltwo$ acts on $\V{1}$  by derivations, and we have the following decomposition of $\V{1}$ as $\sltwo \otimes L_{-1}(\sltwo)$-module:
$$\V{1} = \bigoplus_{n=0} ^{\infty} \rho_n \otimes \LL{n}{1}. $$
\end{proposition}

The algebra $\V{1}$ has a natural ${\Z}_2$--gradation:
$$ \V{1}= \V{1}_0 \oplus \V{1}_1 $$
such that
$  \V{1} _0$ is a simple  \voa:
$$  \mathcal V^{(1)} _0 =  \left(  \mathcal V^{(1)} \right) ^{\Z_2}  = \bigoplus_{n=0} ^{\infty} \rho_{2n} \otimes \LL{2n}{1},  $$
and $  \V{1}_1$ is an irreducible $  \V{1}_0$--module
$$  \V{1}_1 =   \bigoplus_{n=0} ^{\infty} \rho_{2n+1} \otimes \LL{2n+1}{1},  $$

The abelian intertwining algebra $\V{1}$ is also a  building block for a realization of $L_1(\mathfrak{psl}(2 \vert 2) )$.

\begin{proposition} \cite[Remark 9.11]{CG} As \voas
$$ L_1(\mathfrak{psl}(2 \vert 2)) \cong \V{1}_0  \otimes L_1 (\sltwo) \bigoplus \V{1}_1 \otimes  L_1 (\omega_1) .  $$
\end{proposition}

The algebra $\R{1}$ is defined to be
$$ \R{1} = \left( \V{1} \otimes F_{- \frac{1}{2}} \right)^{U(1)} $$
and it can be identified with $M \otimes M$, see e.g. Proposition 5.1 of \cite{C}.
\begin{proposition}  As \voas
$$ \R{1}  \cong M \otimes M.$$
\end{proposition}

\section{Tensor category of $\lksl$}\label{sec:6}

\subsection{Tensor category of orbifold vertex operator algebra}
We recall the following results on the tensor category theory of orbifold vertex operator algebras from \cite{M}.

Let $A$ be an abelian group and $(F, \Omega)$ be a normalized abelian $3$-cocycle on $A$ with values in $\CC^{\times}$, let $V$ be a simple abelian intertwining algebra (\cite{DL1}, \cite{DL2}), a kind of generalized vertex operator algebra graded by $A$ with the usual associativity and commutativity properties of the vertex operator algebra modified by the abelian 3-cocycle $(F, \Omega)$. Let $G$ be a compact Lie group of continuous automorphisms of $V$ containing $\widehat{A}$, let ${\rm Rep}_{A, F, \Omega}(G)$ be the modified tensor category of ${\rm Rep}(G)$ by the $3$-cocycle $(F, \Omega)$.
\begin{theorem}[\cite{M}]\label{mcrae}${}$

\begin{itemize}
\item[(1)]As a $G \otimes V^G$-module, $V$ is semisimple with the decomposition:
\[
V = \bigoplus_{\chi \in \widehat{G}}M_{\chi}\otimes V_{\chi},
\]
where the sum runs over all finite-dimensional irreducible characters of $G$, $M_{\chi}$ is the finite dimensional irreducible $G$-module corresponding to $\chi$, and the $V_{\chi}$ are nonzero, distinct, irreducible $V^G$-modules.
\item[(2)]Let $\mathcal{C}_V$ be the category of $V^G$-modules generated by the $V_{\chi}$. If $V^G$ has a braided tensor category of modules that contains all $V_{\chi}$, then there is a braided tensor equivalence $\mathcal{C}_V \cong {\rm Rep}_{A, F, \Omega}(G)$.
\end{itemize}
\end{theorem}

\subsection{Rigidity of $\rm{KL}_k(\sltwo)$}

Note that the doublet vertex algebra $\A{p}$ is the kernel of the screening operator $\widetilde{Q}$ on the abelian intertwining algebra $V_{\Z \frac{1}{2}\gamma}$, and the vertex algebra $\V{p}$ is the kernel of the screening operator $S$ on $\A{p}\otimes \Pi^{\frac{1}{2}}$. They inherit the structure of the abelian intertwining algebra from $V_{\Z \frac{1}{2}\gamma}$, i.e. we have:
\begin{lemma}
The vertex algebra $\V{p}$ has an abelian intertwining algebra structure with an $SL(2, \mathbb{C})$-action. Moreover,
as an $SL(2, \mathbb{C}) \otimes \lksl$-module,
$$\V{p} \cong  \bigoplus_{n=0} ^{\infty} \rho_n \otimes \LL{n}{p}. $$
\end{lemma}
\begin{corollary} \label{cor:tc} Let $k=-2+\frac{1}{p}$ for $p \in \Z_{\geq 1}$. Then
$\rm{KL}_k \cong {\rm Rep}_{A, F, \Omega}(SU(2))$ as braided tensor categories for some abelian $3$-cocycle $(F, \Omega)$ and in particular $\rm{KL}_k$ is rigid.
\end{corollary}
\begin{proof}
From \cite{K}, the category $\rm{KL}_k$ of ordinary $L_k(\sltwo)$-modules is semisimple with simple objects $\mathcal{L}_n$ for $n \in \Z_{\geq 0}$. Furthermore, it was shown in \cite{CY} that $\rm{KL}_k$ has a braided tensor category structure. Thus by Theorem \ref{mcrae} that $\rm{KL}_k \cong {\rm Rep}_{A, F, \Omega}(SU(2))$ as braided tensor categories for some abelian $3$-cocycle $(F, \Omega)$ and in particular $\rm{KL}_k$ is rigid.
\end{proof}
Let $\rm{KL}_k^{\rm{even}}$ be the subcategory of $\rm{KL}_k$ whose simple objects are the $\LL{2n}{p}$ with $n$ in $\Z_{\geq 0}$. Then the $SL(2, \C)$-action on $\V{p}$ induces a $PSL(2, \C)$-action on the $\Z_2$-orbifold $(\V{p})^{\Z_2}$ and the orbifold decomposes accordingly
\[
(\V{p})^{\Z_2} \cong \bigoplus_{n=0}^\infty \rho_{2n} \otimes \LL{2n}{p}.
\]
The orbifold $(\V{p})^{\Z_2}$ is a \voa{} and hence there is no need for any abelian cocycle, i.e.
\begin{corollary}\label{cor:tceven}
$\rm{KL}_k^{\rm{even}} \cong \rm{Rep}(SO(3))$  as symmetric tensor categories.
\end{corollary}
\begin{remark}\label{Virrem}
Consider the case of $\A{1}$, then $(\A{1})^{SU(2)}$ is $L^{\rm{Vir}}(1,0)$ and its representation category is braided equivalent to ${\rm Rep}_{A, F, \Omega}(SU(2))$ for some abelian $3$-cocycle $(F, \Omega)$ \cite[Example 4.11]{M}. The $\Z_2$ orbifold of $\A{1}$ is just $L_1(\sltwo)$ and it decomposes as
\[
L_1(\sltwo) \cong   \bigoplus_{n=0}^\infty \rho_{2n} \otimes L ^{\rm{Vir}} \left(1, n^2\right)
\]
as $SO(3) \otimes L^{\rm{Vir}}(1,0)$-module. It thus follows that the category of Virasoro modules whose simple objects are the $L ^{\rm{Vir}} \left(1, n^2\right)$ is equivalent to $\rm{Rep}(SO(3))$  as symmetric tensor categories.
\end{remark}

\subsection{Simple currents}\label{sec:sc}

Recall that we have the two $U(1)$-orbifolds ${\V{p}}^{U(1)}$ and $\widetilde{ \R{p}} = \left(\V{p} \otimes F_{-\frac{p}{2}} \right)^{U(1)}$ with corresponding decompositions
\begin{equation}
 \begin{split}
 \V{p} &\cong  \bigoplus_{\ell\in \Z} \V{p}_\ell \otimes \C_\ell, \\
\V{p} \otimes F_{-\frac{p}{2}} &\cong \bigoplus_{\ell\in \Z} \widetilde{\R{p}}_\ell \otimes \C_\ell
 \end{split}
 \end{equation}
 with
  \begin{equation}
 \begin{split}
\V{p}_\ell  &\cong \bigoplus_{s=0}^{\infty} \LL{\vert \ell  \vert + 2 s }{p} \\
\widetilde{ \R{p}}_\ell &=  \bigoplus_{n \in \Z}  \V{p}_{\ell +n} \otimes M_\varphi(1, -n) \\ \end{split}
 \end{equation}
We can apply Theorem \ref{mcrae}, since $\V{p}$ and  $\V{p} \otimes F_{-\frac{p}{2}}$ are simple \voas{} by Proposition \ref{simple} and since $\rm{KL}_k$ as well as Fock modules of the Heisenberg \voa{} have tensor category structure (\cite{CY} and \cite[Theorem 2.3]{CKLR})
and hence also their extensions have this property \cite{CKM}. Especially we have the following fusion rules
\begin{corollary}
The modules $\V{p}_\ell$ are simple currents for the \voa{} ${\V{p}}^{U(1)}$ and the $\widetilde{ \R{p}}_\ell$ are simple currents for $\widetilde{ \R{p}}$. Especially the fusion rules
\[
\V{p}_\ell \boxtimes \V{p}_{\ell'} \cong \V{p}_{\ell + \ell'}\qquad \text{and} \qquad
\widetilde{ \R{p}}_\ell \boxtimes \widetilde{ \R{p}}_{\ell'} \cong \widetilde{ \R{p}}_{\ell + \ell'}
\]
hold.
\end{corollary}
We can thus use the theory of infinite order simple current extensions \cite{CKL} and have  that $\widetilde{ \R{p}}$ is an infinite order simple current extension of ${\V{p}}^{U(1)} \otimes M_\varphi(1)$. Moreover it is generated by the fields corresponding to the generating simple currents, that is the fields in
$\V{p}_{1} \otimes M_\varphi(1, -1)$ and $\V{p}_{-1} \otimes M_\varphi(1, 1)$  together with the ones of
${\V{p}}^{U(1)} \otimes M_\varphi(1)$. But this is exactly how we introduced $\R{p}$, so that we can conclude that
\begin{corollary}\label{r=rtilde}The vertex operator algebras
\[
\widetilde{ \R{p}} ={ \R{p}}.
\]
\end{corollary}

\subsection{Uniqueness of $\V{p}_0$ and $\R{p}$}

Consider the case $p=1$, so that $(\A{1})^{U(1)}= \A{1}_0 \cong M(1)$ is nothing but the simple rank one Heisenberg \voa. The latter is strongly generated by the weight one field and this means that there is a unique simple \voa{} structure up to isomorphism. Recall that there is a one-to-one correspondence between commutative and associative algebras in a vertex tensor category $\mathcal C$ of a \voa{} $V$ and \voa{} extensions of $V$ in $\mathcal C$ \cite{HKL}. Moreover the vertex tensor category of modules of the extended \voa{} that lie in $\mathcal C$ is braided equivalent to the category of local modules for the algebra object that lie in $\mathcal C$ \cite{CKM}. Especially the extended \voa{} is simple if and only if the corresponding algebra object is simple as a module for itself.

The braided equivalence $\rm{KL}_k^{\text{even}} \cong \rm{Rep}(SO(3))$ together with the unique simple \voa{} structure on $\A{1}_0$ and the fact that
the category of Virasoro modules whose simple objects are the $L^{\rm{Vir}} \left(1, n^2\right)$ is equivalent to $\rm{Rep}(SO(3))$ as symmetric tensor categories (by Remark \ref{Virrem})
imply that
\begin{theorem}\label{thm:uniqueness} ${}$
\begin{enumerate}
\item
Up to isomorphism, there is a unique commutative and associative algebra structure on the object
\[
A = \bigoplus_{s=0}^\infty \rho_{2s}
\]
in $\rm{Rep}((SO(3)))$ such that $A$ is simple as a module for itself.
\item For every $p$ in $\Z_{\geq 1}$ and up to isomorphism there is a unique simple \voa{} structure on
\[
\bigoplus_{s=0}^\infty \LL{2s}{p}.
\]
\end{enumerate}
\end{theorem}
\begin{corollary}\label{cor:uniquessRp} For $p$ in $\Z_{\geq 1}$ and $k=-2+\frac{1}{p}$, let $\mathcal X$ be a simple \voa, such that
\[
\mathcal X  \cong \bigoplus_{n \in \Z}  \V{p}_{n} \otimes M_\varphi(1, -n)
\]
as $\lksl \otimes M(1)$-module. Then $\mathcal X \cong \R{p}$.
\end{corollary}
 \begin{proof}
 The Heisenberg coset of $\mathcal X$ is
 \[
\rm{Com}\left( M_\varphi(1), \mathcal X\right) \cong \V{p}_0.
 \]
 This is an isomorphism of \voas{} by Theorem \ref{thm:uniqueness}. 
  Simple current extensions to $\frac{1}{2}\Z$-graded vertex operator (super)algebras are unique up to isomorphism by combining \cite[Proposition 2.15]{CGR}  with \cite[Remark 3.11]{CKL}. Hence any two simple current extensions of $ \V{p}_0 \otimes M(1)$ of the form $\bigoplus\limits_{n \in \Z}  \V{p}_{n} \otimes M_\varphi(1, -n)$ must be isomorphic as \voas.
  \end{proof}
  There is a similar uniqueness result for $\B{p}$ and $\A{p}_0$-algebras. For this we however need to assume the existence of vertex tensor category structure on the Virasoro modules appearing in $\B{p}$. This existence result is work in progress \cite{CFJRY}.
  \begin{theorem}{\rm{\cite{CFJRY}}}\label{CFJRY}
  For any complex number $c$, the category of $C_1$-cofinite modules for $L^{\text{Vir}}(c,0)$ has a vertex tensor category structure.
  \end{theorem}

Let $\mathcal O_{1, p}$ be the category of $L^{\text{Vir}}(c_{1, p},0)$-modules whose objects are direct sums of the simple objects $L^{\text{Vir}}(c_{1, p}, h_{1, n})$
and $\mathcal O_{1, p}^{\text{even}}$ be the subcategory of $\mathcal O_{1, p}$ whose objects are direct sums of the simple objects $L^{\text{Vir}}(c_{1, p}, h_{1, 2n+1})$. By Theorem \ref{CFJRY}, these two categories have braided tensor category structure. Since
 \[
 \A{p} \cong  \bigoplus_{n=0}^\infty \rho_n \otimes L^{\text{Vir}}(c_{1, p}, h_{1, n+1})
 \]
 as $SU(2) \otimes L^{\text{Vir}}(c_{1, p}, 0)$-module, we have in analogy to Corollaries \ref{cor:tc} and \ref{cor:tceven}:
 \begin{corollary} \label{cor:tcVir} For $p$ in $\Z_{\geq 1}$,
the category $\mathcal O_{1, p} \cong {\rm Rep}_{A, F, \Omega}(SU(2))$ as braided tensor categories for some abelian $3$-cocycle $(F, \Omega)$ and in particular $\mathcal O_{1, p}$ is rigid. Moreover,
$\mathcal O_{1, p}^{\rm{even}} \cong \rm{Rep}(SO(3))$  as symmetric tensor categories.
\end{corollary}
 By Theorem \ref{thm:uniqueness} there is a unique simple \voa{} structure on $\bigoplus_{n=0}^\infty  L^{\text{Vir}}(c_{1, p}, h_{1, 2n+1})$  and so any simple \voa{} that is isomorphic to the singlet algebra $\A{p}_0$ as a module for the Virasoro algebra is isomorphic to the singlet algebra $\A{p}_0$ as a \voa.
 The same argument as for Corollary \ref{cor:uniquessRp} applies and we also get uniqueness of $\B{p}$-algebras:
 \begin{corollary}\label{cor:uniquessBp}
 Let $\mathcal Y$ be a simple \voa{} that is isomorphic to $\B{p}$ as a module for the Virasoro algebra times the Heisenberg algebra. Then $\mathcal Y \cong \B{p}$.
 \end{corollary}
We remark that $h_{1, n}>h_{1, m}$ for $n<m$ and hence it follows inductively that any module $M$ for $L^{\text{Vir}}(c_{1, p},0)$ whose character (graded by conformal weight) coincides with  $\bigoplus_{n=0}^\infty  L^{\text{Vir}}(c_{1, p}, h_{1, 2n+s})$ is actually isomorphic as a Virasoro module to  $\bigoplus_{n=0}^\infty  L^{\text{Vir}}(c_{1, p}, h_{1, 2n+s})$. This in turn means that above corollary can be improved to:
\begin{corollary}\label{cor:uniquenessBp}
 Let $\mathcal Y$ be a simple \voa{} that is a module for $L^{\text{Vir}}(c_{1, p},0)$ times a rank one Heisenberg algebra, such that the character graded by conformal weight and Heisenberg weight of $\mathcal Y$ coincides with the one
 of $\B{p}$. Then $\mathcal Y \cong \B{p}$.
 \end{corollary}
As a final remark in this section, let us note that a similar uniqueness Theorem applies for the triplet algebras $\W{p}$. For this let $\mathcal Z$ be a simple \voa{} of central charge $c_{1, p}$. Assume that $U(1)$ is a subgroup of the automorphism group of $\mathcal Z$ so that
\[
\mathcal Z \cong \bigoplus_{n \in \mathbb Z} \mathbb C_n \otimes \mathcal Z_n
 \]
 as $U(1) \otimes \mathcal Z^{U(1)}$-module. Then the $\mathcal Z_n$ are all simple $\mathcal Z^{U(1)}$-modules by \cite{DLM} and the $\mathcal Z_n$ are simple currents by \cite[Thm. 3.1]{CKLR} (since vertex tensor category assumption is satisfied by \cite{CFJRY} together with \cite{CKM}). We assume that the character of $\mathcal Z$, graded by $U(1)$ weight and also by conformal weight coincides with the corresponding graded character of $\W{p}$. Then as just remarked this means that the two algebras are already isomorphic as $U(1) \otimes \text{Vir}(c_{1, p}, 0)$-modules. Especially our uniqueness Theorem thus implies that $\mathcal Z^{U(1)} =\mathcal Z_0$ is isomorphic to the singlet \voa{} $\A{p}_0$ as a \voa. Uniqueness of simple current extensions then implies that $\mathcal Z \cong \W{p}$ as \voas. We summarize:
 \begin{corollary}
 Let $\mathcal Z$ be a simple \voa{} of central charge $c_{1, p}$ with $U(1)$ as subgroup of automorphism such that the character, graded by $U(1)$ weight and conformal weight, coincides with the graded character of $\W{p}$. Then $\mathcal Z\cong \W{p}$ as \voas.
 \end{corollary}

\section{$\cW$-algebras and chiral algebras of Argyres-Douglas theories}\label{sec:AD}

Let $\mathfrak{g}=\mathfrak{sl}_n$ and let $f$ be a nilpotent element corresponding to the partition $(n-2, 1, 1)$ of $n$. This means that there exists an $\sltwo$-triple $e, h, f$ in $\mathfrak{sl}_n$ such that the standard representation of $\mathfrak{sl}_n$ decomposes as $\rho_{n-2} \oplus \rho_1 \oplus \rho_1$ under this $\sltwo$-action. We denote the corresponding simple $\cW$-algebra at level $\ell$ by $\cW_\ell(\mathfrak{sl}_n, f)$. It contains an affine subalgebra of $\mathfrak{gl}_{2}$ at level $k=\ell+n-3$ as subalgebra. We set $p=n-1$ and $\ell = -\frac{p^2-1}{p}$ so that $k=\ell + p -2 = -2+ \frac{1}{p}$. Note that $\ell+n = \frac{n}{n-1}$ is a boundary admissible level.
The main result of \cite[Theorem 5.7]{C} tells us that for this choice $\cW_\ell(\mathfrak{sl}_{p+1}, f) \cong \R{p}$ as $\lksl \otimes M(1)$-modules. By Corollary \ref{cor:uniquessRp} we have that
 \begin{theorem}\label{thm:RpW} Let $\ell = -\frac{p^2-1}{p}$. Then
 $\cW_\ell(\mathfrak{sl}_{p+1}, f) \cong \R{p}$ as \voas.
 \end{theorem}
 Argyres-Douglas theories are four dimensional $\mathcal N=2$ superconformal field theories. They have associated chiral algebras, that are actually \voas{} \cite{Beemetal}. The theories are characterized by pairs of Dynkin diagrams and we are interested in the case of $(A_1, D_{2p})$, see \cite{BN1, CS}. Usually not much is known about these chiral algebras. However in this case, the Schur-index, that is the character of the chiral algebra is known and it agrees with the character of $\cW_\ell(\mathfrak{sl}_{p+1}, f)$ by \cite[Theorem 5.7]{C}. Moreover the flavour symmetries of these algebras are known and in this case the translation is that the chiral algebra is an extension of $\lksl \otimes M(1)$ with $k=-2+\frac{1}{p}$. The question of simplicity has not appeared in the physics literature yet, but it is a natural requirement that chiral algebras of these gauge theories will most often be simple.
 Let us summarize
 \begin{remark}
 The chiral algebra $\mathcal X$ of $(A_1, D_{2p})$ Argyres-Douglas theories has the properties
 \begin{enumerate}
 \item $\mathcal X$ is a simple \voa,
 \item $\mathcal X$ is an extension of $\lksl \otimes M(1)$ with $k=-2+\frac{1}{p}$.
 \item The character of $\mathcal X$ coincides with the one of $\R{p}$.
 \end{enumerate}
 \end{remark}
 \begin{theorem}
$\R{p}$ is the chiral algebra of $(A_1, D_{2p})$ Argyres-Douglas theories.
 \end{theorem}
 Analogous results hold for the $\B{p}$-algebra. Theorem 27 of \cite{ACKR} says that the character of $\B{p}$ and of the simple subregular $\cW$-algebra of $\mathfrak{sl}_{p-1}$,  $\cW_\ell(\mathfrak{sl}_{p-1}, f_{\text{sub}})$,
 at level $\ell = 1-p +\frac{p-1}{p}$ coincide. The central charges of the Virasoro subalgebras also coincide so that Corollary \ref{cor:uniquenessBp} applies, i.e.
 \begin{theorem} \label{thm:BpW} Let $\ell = -\frac{(p-1)^2}{p}$. Then
 $\cW_\ell(\mathfrak{sl}_{p-1}, f_{\text{sub}}) \cong \B{p}$ as \voas.
 \end{theorem}
 The chiral algebra of $(A_1, A_{2p-3})$ Argyres-Douglas theories is a  \voa{} whose character and central charge coincide with the corresponding data of the $\B{p}$-algebra \cite{CS, BN1, C}. Again it is natural to require it to be a simple \voa.
 \begin{remark}
 The chiral algebra $\mathcal Y$ of $(A_1, A_{2p-3})$ Argyres-Douglas theories has the properties
 \begin{enumerate}
 \item $\mathcal Y$ is a simple \voa,
 \item $\mathcal Y$ is an extension of $L^{\text{Vir}}(c_{1, p}, 0) \otimes M(1)$.
 \item The character of $\mathcal X$ coincides with the one of $\B{p}$.
 \end{enumerate}
 \end{remark}
The uniqueness result (Corollary \ref{cor:uniquenessBp}) tells us that
 \begin{theorem}
$\B{p}$ is the chiral algebra of $(A_1, A_{2p-3})$ Argyres-Douglas theories.
 \end{theorem}

\section{Quantum Hamilton reduction}\label{sec:8}

The aim of this section is to prove that $\B{p}$ is as a \voa{} the quantum Hamiltonian reduction of $\R{p}$ and also that $\A{p}$ is the quantum Hamiltonian reduction of $\V{p}$. For this we need to understand the reduction of modules of $\Pi(0)$ first.

\subsection{Reduction of  $\Pi(0)$-modules }

Recall that $\Pi(0) = M (1) \otimes {\C}[\Z c]$, and we have modules $\Pi_r (\lambda) = \Pi(0) . e^{ r \mu + \lambda c}$. We
 set $\Pi (\lambda) := \Pi_ 0 (\lambda)$.

In this section we shall introduce the Drinfeld-Sokolov reduction cohomology for $V(\mathfrak{n})$-modules and apply for modules which we constructed above, where $\mathfrak{n} = \C e \subset \sltwo$ and $V(\mathfrak{n})$ is the vertex subalgebra of $V_k(\sltwo)$ generated by $e(z)$. Let $F$ be the fermionic vertex superalgebra generated by the fields
$$ \Psi^{+}(z) = \sum_{n \in {\Z} } \Psi^+ (n) z^{-n-1}, \ \Psi^{-} (z) = \sum_{n \in {\Z} } \Psi^{-} (n)  z ^{-n}$$ such that the components of the fields $\Psi^{\pm} (z)$ satisfy the anti-commutation relation for the Clifford algebra
$$ \{ \Psi^{\pm} (n), \Psi ^{\pm} (m) \} = 0, \quad \{ \Psi^{\pm} (n), \Psi^{\mp} (m) \} = \delta_{n+m,0} \quad (n,m \in {\Z}). $$
The conformal vector
$$\omega_{fer} = \Psi^{-} (-1) \Psi^{+} (-1) {\vak}$$
defines on $F$ the structure of a vertex operator superalgebra with central charge $c=-2$.
The fermionic vertex superalgebra $F$ has the charge decomposition by $\operatorname{charge}(\Psi^\pm (n)) = \mp 1$, and we denote by $F^i$ the homogeneous subspace of $F$ with the charge degree $i$. We have $F = \bigoplus_{i \in \Z} F^i$. By using the boson-fermion correspondence $F$ can be realized as lattice vertex superalgebra
\begin{align*}
F \xrightarrow{\sim} V_{\Z \phi} = M_{\phi} (1) \otimes {\C}[\Z \phi],\quad
\Psi^\pm \mapsto e^{\mp \phi},\quad
\mbox{ where} \quad  \la \phi, \phi \ra = 1,
\end{align*}
and conformal vector is given by
$$\omega_{fer} = \frac{1}{2} ( \phi(-1) ^2 - \phi(-2) ) {\vak}.$$

Given a $V(\mathfrak{n})$-module $M$, set the complex $C(M) = M \otimes F$ and the differential $d_{\mathrm{DS} (0)}$, where
\begin{align*}
d_{\mathrm{DS}} = (e+1) \otimes e^{\phi} \in V(\mathfrak{n}) \otimes F.
\end{align*}
Then $C(M) = \bigoplus_{i \in \Z}C^i(M)$, where $C^i(M) = M \otimes F^i$. Since $d_{\rm{DS} (0)}\cdot C^i(M) \subset C^{i+1}(M)$ and $d^2_{\rm{DS} (0)} = 0$, the pair $(C(M), d_{\rm{DS} (0)})$ forms the cochain complex. The Drinfeld-Sokolov reduction $H_{\rm{DS}}^\bullet(M)$ for $M$ is defined by
\begin{align*}
H_{\mathrm{DS}}^\bullet(M) = H^\bullet(C(M), d_{\mathrm{DS} (0)}).
\end{align*}
Since $d_{\mathrm{DS} (0)}$ is a vertex operator of $d_{\mathrm{DS}}$, if $V$ is a vertex superalgebra with a map $V(\mathfrak{n}) \rightarrow V$ of vertex superalgebras, $H_{\mathrm{DS}}(V)$ inherits a vertex superalgebra structure from that of $C(V)$. Moreover, if $M$ is a $V$-module, $H_{\mathrm{DS}}(M)$ is a $H_{\mathrm{DS}}(V)$-module.

Consider now the Drinfeld-Sokolov reductions for the vertex algebra $\Pi(0)$ with a vertex algebra homomorphism
\begin{align*}
V(\mathfrak{n}) \ni e \mapsto e^c \in \Pi(0).
\end{align*}
and its irreducible modules $\Pi(\lambda)$, $\Pi_r (\lambda)$.

\begin{proposition} \label{pom-1} ${}$
\begin{enumerate}
\item $ H_{\mathrm{DS}}^i(\Pi(0)) = \delta_{i,0}{\C} {\vak}$,
\item $H_{\mathrm{DS}}^i (\Pi(\lambda) ) = \delta_{i,0}{\C} e ^{\lambda(\alpha+ \beta)}$,
\item $H_{\mathrm{DS}}  (\Pi _r ( \lambda) )  = 0 $  if $   r \ne 0$.
\end{enumerate}
\end{proposition}
\begin{proof}
Let
\begin{align*}
\omega_{C(\Pi(0))} = \frac{1}{2}  c(-1) d(-1) + \omega_{fer}.
\end{align*}
Then $\omega_{C(\Pi(0))}$ is the conformal vector of $C(\Pi(0))$ with the central charge $0$. The conformal weights of $c, d, e^c, e^{-c}, e^\phi, e^{-\phi}$ are 1, 1, 0, 0, 1, 0 respectively. Denote by $L^{\Pi}(0)$ the gradation operator with respect to $\omega_{C(\Pi(0))}$, and by $C(\Pi(0))_l$ the homogeneous subspace of $C(\Pi(0))$ with the conformal weight $l \in \Z$. Then $C(\Pi(0)) = \bigoplus_{l \in \Z} C(\Pi(0))_l$. Now, we have
\begin{align*}
\omega_{C(\Pi(0))} = d_{\mathrm{DS}(0)} \cdot
\frac{1}{2}\left(  -d(-1) \phi (-1) + \phi(-1) ^2 - \phi(-2) \right) e^{-c - \phi} \in \operatorname{Im}d_{\mathrm{DS} (0)}.
\end{align*}
Since $\operatorname{Im}d_{\mathrm{DS} (0)}$ is an ideal of $\operatorname{Ker}d_{\mathrm{DS} (0)}$, if $v$ is a vector in $\operatorname{Ker}d_{\rm{DS} (0)} \cap C(\Pi(0))_l$ with $l \neq 0$, we have
\begin{align*}
v = \frac{1}{l}L^{\Pi}(0)v \in \operatorname{Im}d_{\mathrm{DS} (0)}.
\end{align*}
Thus, $H_{\mathrm{DS}}(\Pi(0)) = H(C(\Pi(0))_0, d_{\mathrm{DS} (0)})$. Using the facts that
\begin{align*}
C(\Pi(0))_0 = \operatorname{Span}\{e^{i c}, e^{i c - \phi} \mid i \in \Z\}
\end{align*}
and that
\begin{align*}
d_{\mathrm{DS}(0)} \cdot e^{i c - \phi} = e^{i c} + e^{(i+1)c} \neq 0,\quad
d_{\mathrm{DS}(0)} \cdot e^{i c} =0,
\end{align*}
it follows that
\begin{align*}
&\operatorname{Ker}d_{\mathrm{DS} (0)} \cap C(\Pi(0))_0 = \operatorname{Span}\{e^{i c} \mid i \in \Z\},\\
&\operatorname{Im}d_{\mathrm{DS} (0)} \cap C(\Pi(0))_0 = \operatorname{Span}\{e^{i c} + e^{(i+1)c} \mid i \in \Z\}.
\end{align*}
Hence
\begin{align*}
H_{\mathrm{DS}}(\Pi(0)) = H^0_{\mathrm{DS}}(\Pi(0)) = \frac{\operatorname{Ker}d_{\mathrm{DS} (0)} \cap C(\Pi(0))_0}{\operatorname{Im}d_{\mathrm{DS} (0)} \cap C(\Pi(0))_0} = \C \vak.
\end{align*}
Similarly we get $H_{\mathrm{DS}} (\Pi(\lambda) )  = \C e ^{\lambda(\alpha+ \beta)}$. This proves assertion (1) and (2).

Next we notice that $c(-1) \in  \operatorname{Ker}d_{\mathrm{DS} (0)}$, and since it has conformal weight $1$, it is in   $\operatorname{Im}d_{\mathrm{DS} (0)}$. Since $\operatorname{Im}d_{\mathrm{DS} (0)}$ is an ideal in  $\operatorname{Ker}d_{\mathrm{DS} (0)}$, we conclude that  for each $w \in \Pi_r (\lambda)$:

$$ w = \frac{1}{r} c(0) w  \in \operatorname{Im}d_{\mathrm{DS} (0)}.$$
This proves  assertion (3).
\end{proof}

Let $c \in \C$,  $U_1$  any $V^{\mathrm{Vir}}(c,0)$--module and  $U^{2}$ is any $L ^{\mathrm{Vir}} (c,0)$--module. As above, we will identify $U_i \otimes \Pi(0)$ with a $V(\mathfrak{n})$-module only acting on the second factor $\Pi(0)$. Using Proposition \ref{pom-1}, the following is clear:

\begin{corollary} \label{notes} ${}$
\begin{enumerate}
\item Assume that $k \in {\C} \setminus \{-2 \}$ and that $U_1$ is any $V ^{\rm{Vir}} (c_k,0)$--module. Then
$$ H ^{i}_{\rm{DS}} (V ^{\rm{Vir}} (c,0) \otimes \Pi(0)) \cong  \delta_{i,0}  V^{\rm{Vir}} (c,0)$$
and  $$H ^{i}_{\rm{DS}} (U_1 \otimes \Pi(\lambda)) \cong  \delta_{i,0} U_1 $$
as $V ^{\rm{Vir}}(c_k,0)$--modules.

\item Assume that $V$ is any \voa{} extension of $V ^{\rm{Vir}} (c_k,0)$. Then as $V ^{\rm{Vir}} (c_k,0)$--modules
$$ H ^{i}  _{\rm{DS}} ( V  \otimes \Pi (0))   =\delta_{i,0}  V. $$
\end{enumerate}
\end{corollary}

\subsection{Quantum Hamilton reduction of $\mathcal V^{(p)}$ and $\mathcal R^{(p)}$ }

In this section we will prove
\begin{theorem}\label{reduction}
As vertex operator algebras,
\[
H_{\mathrm{DS}}^0(\R{p}) = \B{p},
\]
and as abelian intertwining algebras,
\[
H_{\mathrm{DS}}^0(\V{p}) = \A{p}.
\]
\end{theorem}

First we need to recall some known statements
\begin{lemma}\label{reduction1}
As vertex operator algebras,
\[
H_{\mathrm{DS}}^0(V^k(\sltwo)) \cong V^{\mathrm{Vir}}(c_{1,p}, 0).
\]
As $V^{\mathrm{Vir}}(c_{1,p}, 0)$-modules
\begin{align*}
H_{\mathrm{DS}}^0(\mathcal{L}_s^{(p)}) \cong L^{\mathrm{Vir}} (c_{1,p},  h_{1,s+1} );\quad
H_{\mathrm{DS}}^i(\mathcal{L}_s^{(p)}) = 0,\ i \neq 0
\end{align*}
and
\begin{align*}
&H_{\rm{DS}}^0(\R{p}) \cong \B{p},\;\;\; H_{\rm{DS}}^0(\V{p}) \cong \A{p},\quad
&H_{\rm{DS}}^i(\R{p})=H_{\rm{DS}}^i(\V{p})=0,\ i \neq 0.
\end{align*}
\end{lemma}
\begin{proof}
The first statement follows from \cite{FF90}, the second one follows from Theorem 6.7.1 and Theorem 6.7.4 in \cite{Ar05}, and the last one follows from the decomposition of $\R{p}$, $\B{p}$, $\V{p}$ and $\A{p}$ as Virasoro algebra $V^{\mathrm{Vir}}(c_{1,p}, 0)$-modules. See also Proposition 5.10 of \cite{C}.
\end{proof}

By Proposition \ref{prop:VpviaAp}, we have the short exact sequence of $\V{p}$-modules:
\begin{align}\label{short}
0 \rightarrow \V{p} \xrightarrow{\widetilde{\phi}} (\A{p} \otimes \Pi(0)^{\frac{1}{2}})^{\Z_2} \stackrel{S}{\rightarrow} \mathrm{Im}(S) \rightarrow 0.
\end{align}

Then the cohomology functor $H^\bullet_{\mathrm{DS}}(?)$ yields the long exact sequence of $H^0_{\mathrm{DS}}(\V{p}$-modules from the exact sequence \eqref{short}. From Lemma \ref{reduction1} as Virasoro algebra $V^{\mathrm{Vir}}(c_{1,p}, 0)$-modules
\[
H^i_{\mathrm{DS}}(\V{p}) \cong \delta_{i,0}\A{p}.
\]
Also, it follows from  Corollary  \ref{notes} and both of equations \eqref{eq:ApPiZ2-dec} and $\Pi(0)^{\frac{1}{2}}_1 = \Pi(\frac{1}{2})$ that as Virasoro algebra $V^{\mathrm{Vir}}(c_{1,p}, 0)$-modules
\[
H^i_{\mathrm{DS}}\left( (\A{p}\otimes \Pi(0)^{\frac{1}{2}})^{\Z_2}\right) \cong \delta_{i,0}\A{p}.
\]
Thus, the long exact sequence induced from \eqref{short} gives rise to the exact sequence of $V^{\rm{Vir}}(c_{1,p}, 0)$-modules
\begin{align}\label{long}
0 \rightarrow H^{-1}_{\mathrm{DS}}(\mathrm{Im}(S)) \rightarrow H^0_{\mathrm{DS}}(\V{p}) \stackrel{\phi}{\rightarrow} H^0_{\mathrm{DS}}((\A{p}\otimes \Pi(0)^{\frac{1}{2}})^{\Z_2}) \rightarrow H^{0}_{\mathrm{DS}}(\mathrm{Im}(S)) \rightarrow 0,
\end{align}
and the vanishing results
\begin{align}\label{vanish-ImS}
H^{i}_{DS}(\operatorname{Im}(S)) = 0,\ i \neq 0, -1.
\end{align}
$\phi$ is a homomorphism $\phi: \A{p} \rightarrow \A{p}$ of abelian intertwining algebras and it actually is an isomorphism:
\begin{lemma}\label{lemma1}
$\phi: \A{p} \rightarrow \A{p}$ is an isomorphism of abelian intertwining algebras.
\end{lemma}
\begin{proof}
As $\widetilde{\phi}$ in \eqref{short} is a homomorphism of abelian intertwining algebras, so is $\phi$. Thus, it is enough to show that $\phi$ is an isomorphism of $L^{\rm{Vir}}(c_{1,p}, 0)$-modules. Notice that $\Pi\left(\frac{s}{2}\right)$ for $s \in \Z$ is isomorphic to $\Pi(0)=\Pi(0)^{\frac{1}{2}}_0$ (resp.  $\Pi\left(\frac{1}{2}\right) = \Pi(0)^{\frac{1}{2}}_1$) as a $\Pi(0)$-module if $s$ is even (resp. odd). First, the injective map $\widetilde{\phi}_{s, j} \colon \mathcal{L}_s^{(p)} \hookrightarrow L^{\rm{Vir}} (c_{1,p},  h_{1,s+1} ) \otimes \Pi(\frac{s}{2})$ given in Proposition \ref{prop:Ls-in-LVir} yields an exact sequence
\begin{align}\label{eq:summand for short}
0 \rightarrow \mathcal{L}_s^{(p)} \xrightarrow{\widetilde{\phi}_{s, j}} L^{\rm{Vir}} (c_{1,p},  h_{1,s+1} ) \otimes \Pi\left(\frac{s}{2}\right) \xrightarrow{S_s} \operatorname{Im}S_s \rightarrow 0,
\end{align}
where $\operatorname{Im}S_s = \left( L^{\rm{Vir}} (c_{1,p},  h_{1,s+1} ) \otimes \Pi\left(\frac{s}{2}\right) \right)/\mathcal{L}_s^{(p)}$ and $S_s$ is the canonical projection. Consider the map
\begin{align*}
\phi_{s, j} \colon H^0_{\mathrm{DS}}(\mathcal{L}_s^{(p)}) \rightarrow L^{\mathrm{Vir}} (c_{1,p},  h_{1,s+1} )
\end{align*}
induced from $\widetilde{\phi}_{s, j}$ through the Drinfeld-Sokolov reduction functor $H_{\rm{DS}}(?)$. By Proposition \ref{pom-1}, we have
\begin{align*}
H^i_{\mathrm{DS}}(L^{\mathrm{Vir}} (c_{1,p},  h_{1,s+1} ) \otimes \Pi\left(\frac{s}{2}\right) )
= L^{\mathrm{Vir}} (c_{1,p},  h_{1,s+1} ) \otimes H^i_{\mathrm{DS}}( \Pi\left(\frac{s}{2}\right) )\\
= \delta_{i,0}L^{\mathrm{Vir}} (c_{1,p},  h_{1,s+1} ) = H^i_{\mathrm{DS}}(\mathcal{L}_s^{(p)}).
\end{align*}
Recall that $v_{1, s+1, j} \otimes e^{\frac{s}{2} c}$ is the image of the highest weight vector of $\mathcal{L}_s^{(p)}$ by $\widetilde{\phi}_{s, j}$. Since
\begin{align*}
d_{\rm{DS} (0)} \cdot ( v_{1, s+1, j} \otimes e^{\frac{s}{2} c}) = 0
\end{align*}
for the differential $d_{\rm{DS} (0)}$ of the Drinfeld-Sokolov reduction for $L^{\rm{Vir}} (c_{1,p},  h_{1,s+1} ) \otimes \Pi\left(\frac{s}{2}\right)$, the vector $v_{1, s+1, j} \otimes e^{\frac{s}{2} c}$ is in $\operatorname{Ker} d_{\rm{DS} (0)}$ with degree $0$. Moreover, the vector $v_{1, s+1, j} \otimes e^{\frac{s}{2} c}$ is non-zero in $H^0_{\rm{DS}}(L^{\rm{Vir}} (c_{1,p},  h_{1,s+1} ) \otimes \Pi\left(\frac{s}{2}\right))$ by the proof of Proposition \ref{pom-1}. Thus, the map $\phi_{s, j}$ is a non-trivial homomorphism of $L^{\rm{Vir}}(c_{1,p}, 0)$-modules, and so is an isomorphism. Now, we have the decompositions
\begin{align*}
\V{p} = \bigoplus_{s = 0}^\infty \rho_s \otimes \mathcal{L}_s^{(p)},\quad
\A{p}\otimes \Pi\left(\frac{s}{2}\right) = \bigoplus_{s = 0}^\infty \rho_s \otimes L^{\rm{Vir}} (c_{1,p},  h_{1,s+1} ) \otimes \Pi\left(\frac{s}{2}\right)
\end{align*}
as $\sltwo \otimes \lksl$-modules. Set the projections
\begin{align*}
&p_s \colon \V{p} \twoheadrightarrow \rho_s \otimes \mathcal{L}_s^{(p)},\\
&q_s \colon \A{p} \otimes \Pi\left(\frac{s}{2}\right) \twoheadrightarrow \rho_s \otimes L^{\rm{Vir}} (c_{1,p},  h_{1,s+1} ) \otimes \Pi\left(\frac{s}{2}\right),\\
&r_{s, j} \colon \rho_s \twoheadrightarrow \rho_{s, j},
\end{align*}
where $\rho_{s, j}$ is the weight space of $\rho_s$ with the weight $2j - s$ for $j = 0, \ldots, s$, and let
\begin{align*}
&p_{s, j} = (r_{s, j}\otimes\operatorname{id}_{ \mathcal{L}_s^{(p)} }) \circ p_s,\\
&q_{s, j} = (r_{s, j}\otimes\operatorname{id}_{L^{\rm{Vir}} (c_{1,p},  h_{1,s+1} ) \otimes \Pi\left(\frac{s}{2}\right),}) \circ q_s.
\end{align*}
Using Proposition \ref{prop:Ls-in-LVir}, Lemma \ref{ker} and the proof of Theorem \ref{thm:Vp-dec}, it follows that the top space of $\widetilde{\phi}(\rho_{s, j} \otimes \mathcal{L}_s^{(p)})$ is spanned by $v_{1, s+1, j} \otimes e^{\frac{s}{2} c}$ so that $\widetilde{\phi}(\rho_{s, j} \otimes \mathcal{L}_s^{(p)})$ is embedded into $\rho_{s, j} \otimes L^{\rm{Vir}} (c_{1,p},  h_{1,s+1} ) \otimes \Pi\left(\frac{s}{2}\right)$. Thus, the restriction of the map $\widetilde{\phi}$ to $\rho_{s, j} \otimes \mathcal{L}_s^{(p)}$ factors uniquely through $\widetilde{\phi}_{s, j}$, i.e.
\begin{align*}
\widetilde{\phi}_{s, j} \circ p_{s, j} = q_{s, j} \circ \widetilde{\phi},\quad
s \in \Z_{\geq 0}.
\end{align*}
Hence, the exact sequence \eqref{short} consists of the direct sum of the exact sequence \eqref{eq:summand for short} for all $s, j$. We conclude that
\begin{align*}
\widetilde{\phi} = \bigoplus_{s=0}^\infty \bigoplus_{j = 0}^s \widetilde{\phi}_{s, j},
\end{align*}
which implies that
\begin{align*}
\phi = \bigoplus_{s=0}^\infty \bigoplus_{j = 0}^s \phi_{s, j}.
\end{align*}
As each $\phi_{s, j}$ is an isomorphism of $L^{\rm{Vir}}(c_{1,p}, 0)$-modules, so is $\phi$. This completes the proof.
\end{proof}

As a consequence of Lemma \ref{lemma1}, the exact sequence \eqref{long} and the formulae \eqref{vanish-ImS}, we have
\[
H_{\mathrm{DS}}(\operatorname{Im}(S)) = 0.
\]
Again from the long exact sequence (\ref{long}), the vertex operator algebra homomorphism $\phi$ is an isomorphism. We proved Theorem \ref{reduction} for the algebra $\V{p}$, the proof for $\R{p}$ is similar.

\subsection{Reduction of $L_1 (\mathfrak{psl}(2 \vert 2))$}
\label{red-psl22}
Recall that $$ L_1 (\mathfrak{psl}(2 \vert 2)) = \V{1} _0 \otimes L_1 (\sltwo) \bigoplus  \V{1} _1 \otimes L_1 (\omega_1) . $$
Since $$H_{\rm{DS}} ( \V{1}_0)  = \A{1}_0 = L_1(\sltwo), \quad  H_{\rm{DS}} ( \V{1}_1)  = \A{1} _1 = L_1(\omega_1), $$
we get
$$ H_{\rm{DS}}   ( L_1 (\mathfrak{psl}(2 \vert 2))  ) \cong L_1(\sltwo) ^{\otimes 2} \ \oplus \ L_1(\omega_1 ) ^{\otimes 2}  \cong F_ 1 ^{\otimes 2},$$
where $F_1$ is the Clifford vertex algebra ($bc $ system) of central charge $c=1$.




\subsection*{Acknowledgments}

This work was done in part during the visit of D.A. to the University of Alberta. T.C. appreciates the many discussions with Boris Feigin.

 D.A.   is   partially
supported   by the
QuantiXLie Centre of Excellence, a project cofinanced
by the Croatian Government and European Union
through the European Regional Development Fund - the
Competitiveness and Cohesion Operational Programme
(KK.01.1.1.01.0004).

T. C is supported by NSERC $\#$RES0020460.

N. G is supported by JSPS Overseas Research Fellowships.

\vskip10pt {\footnotesize{}{ }\textbf{\footnotesize{}D.A.}{\footnotesize{}:
Department of Mathematics, Faculty of Science, University of Zagreb, Bijeni\v{c}ka 30,
10 000 Zagreb, Croatia; }\texttt{\footnotesize{}adamovic@math.hr}{\footnotesize \par}

\vskip10pt {\footnotesize{}{ }\textbf{\footnotesize{}T.C.}{\footnotesize{}:
Department of Mathematical and Statistical Sciences, University of Alberta, 632 CAB, Edmonton, Alberta, Canada T6G 2G1; }\texttt{\footnotesize{}creutzig@ualberta.ca}{\footnotesize \par}

\vskip10pt {\footnotesize{}{ }\textbf{\footnotesize{}N.G.}{\footnotesize{}:
Department of Mathematical and Statistical Sciences, University of Alberta, 632 CAB, Edmonton, Alberta, Canada T6G 2G1; }\texttt{\footnotesize{}genra@ualberta.ca}{\footnotesize \par}

 \vskip10pt {\footnotesize{}{ }\textbf{\footnotesize{}J.Y.}{\footnotesize{}:
 Department of Mathematical and Statistical Sciences, University of Alberta, 632 CAB, Edmonton, Alberta, Canada T6G 2G1; }\texttt{\footnotesize{}jinwei2@ualberta.ca}{\footnotesize \par}

\begin{thebibliography}{ABKS}

\bibitem{A-singlet} D. Adamovi\'c, Classification of irreducible modules of certain subalgebras of free boson vertex algebra, J. Alg. {\bf 270} (2003), 115--132.

\bibitem{A} D. Adamovi\'c, A realization of certain modules for the $N=4$ superconformal algebra and the affine Lie algebra $A_2 ^{(1)}$,  Transformation Groups, {\bf 21} (2016), no. 2, 299--327.

\bibitem{ACGY-inverse} D. Adamovi\'c,  T. Creutzig, N. Genra and J. Yang,  Inverse reduction, in preparation.


\bibitem{A-2019} D. Adamovi\'c, Realizations of Simple Affine Vertex Algebras and Their Modules: The Cases $\widehat {sl (2)}$ and $\widehat {osp (1, 2)}$, Comm. Math. Phys. {\bf 366} (2019), no. 3, 1025--1067.




 
\bibitem{AKMP2} D. Adamovi\'c, V. G. Kac, P. M. Frajria, P. Papi and O. Per\u{s}e, Finite vs infinite decompositions in conformal embeddings, Comm. Math. Phys. {\bf 348} (2016), no. 2, 445--473.

\bibitem{AKMP3} D. Adamovi\'c, V. G. Kac, P. M. Frajria, P. Papi and O. Per\u{s}e,  Conformal embeddings of affine vertex algebras in minimal W-algebras I: structural results,  J. Alg. {\bf 500} (2018), 117--152.

\bibitem{AKMP} D. Adamovi\'c, V. G. Kac, P. M. Frajria, P. Papi and O. Per\u{s}e, Conformal embeddings of affine vertex algebras in minimal W-algebras II: decompositions,  Japanese Journal of Mathematics, {\bf 12} (2017), no. 2, 261--315.


\bibitem{ALM} D. Adamovi\' c, X. Lin, A. Milas, ADE subalgebras of the triplet vertex algebra $W(p)$: $A$-series, Comm. Contemp. Math. {\bf 15} (2013), no. 6, 1350028, 30 pages.


 \bibitem{AdM-singlet} D. Adamovi\' c, A. Milas, Logarithmic intertwining operators and $W(2,2p-1)$-algebras,
  J.\ Math.\ Phys.\  {\bf 48} (2007), no. 7, 073503, 20 pages.

\bibitem{AdM-triplet} D.~Adamovi\' c, A. Milas, On the triplet vertex algebra $W(p)$,
  Adv.\ Math.\  {\bf 217} (2008), no. 6, 2664--2699.

\bibitem{AdM-zhu}D.~Adamovi\' c, A. Milas, The structure of Zhu's algebras for certain $\mathcal W$-algebras, Adv. Math. {\bf 227} (2011), no. 6, 2425--2456.

\bibitem{AdM-doublet} D. Adamovi\' c, A. Milas, The doublet vertex operator superalgebras $\mathcal A(p)$ and $\mathcal A_{2,p}$, Contemporary Mathematics {\bf 602} (2013), 23--38.

\bibitem{AdM-LCFT} D. Adamovi\' c,  A. Milas, Vertex operator (super)algebras and LCFT, J. Phys. A {\bf 46} (2013), no. 49, 494005, 23 pages.


\bibitem{AdM-LCFT2} D. Adamovi\' c,  A. Milas,  $C_2$-cofinite vertex algebras and their logarithmic modules, in: Conformal Field Theories and Tensor Categories, Proceedings of a Workshop Held at Beijing International Center for Mathematics Research, ed. C. Bai, J. Fuchs, Y.-Z. Huang, L. Kong, I. Runkel and C. Schweigert, Mathematical Lectures from Beijing University, Vol. 2, Springer, New York, 2014, 249--270.

\bibitem{ACKR}
  J.~Auger, T.~Creutzig, S.~Kanade and M.~Rupert, Braided Tensor Categories related to $\mathcal{B}_p$ Vertex Algebras, arXiv:1906.07212.

\bibitem{ACL} T. Arakawa, T. Creutzig and A. R. Linshaw, W-algebras as coset vertex algebras, Invent. Math. {\bf 218} (2019), no. 1, 145--195.

\bibitem{AD} P. C. Argyres and M. R. Douglas, New phenomena in $SU(3)$ supersymmetric gauge theory, Nucl. Phys. B {\bf 448} (1995), 93--126.

\bibitem{AFO} M.~Aganagic, E.~Frenkel and A.~Okounkov, Quantum $q$-Langlands Correspondence,
  Trans.\ Moscow Math.\ Soc.\  {\bf 79} (2018) 1--83.

\bibitem{Ar05} T. Arakawa, Representation theory of superconformal algebras and the Kac-Roan-Wakimoto Conjecture,
Duke Math. J. {\bf 130} 2005, no. 3, 435--478.

\bibitem{Ar1} T. Arakawa, Representation theory of $W$-algebras, Invent. Math. {\bf 169} (2007), no. 2, 219--320.

\bibitem{ACKL} T. Arakawa, T. Creutzig, K. Kawasetsu and A. R. Linshaw, Orbifolds and cosets of minimal $W$-algebras, Comm.\ Math.\ Phys.\  {\bf 355} (2017) no.1, 339--372.

\bibitem{BN1}
  M.~Buican and T.~Nishinaka, On the superconformal index of Argyres-Douglas theories,
  J.\ Phys.\ A {\bf 49} (2016) no.1, 015401, 33 pages.

\bibitem{BN2}
  M.~Buican and T.~Nishinaka, On irregular singularity wave functions and superconformal indices,
  JHEP {\bf 1709} (2017) 066.

\bibitem{Beemetal}
  C.~Beem, M.~Lemos, P.~Liendo, W.~Peelaers, L.~Rastelli and B.~C.~van Rees, Infinite Chiral Symmetry in Four Dimensions,
  Commun.\ Math.\ Phys.\  {\bf 336} (2015) no.3, 1359-1433.





\bibitem{C}
T. Creutzig, $W$-algebras for Argyres-Douglas theories, {\em Euro. J. Math.} {\bf 3} (2017), no. 3, 659--690.

\bibitem{C2}
T. Creutzig, Fusion categories for affine vertex algebras at admissible levels, {\em Selecta Math. (N.S.)} {\bf 25} (2019), no. 2, Art. 27, 21 pages.

\bibitem{C3}
T.~Creutzig, Logarithmic W-algebras and Argyres-Douglas theories at higher rank,  JHEP {\bf 1811} (2018) 188.

\bibitem{CGR}
  T.~Creutzig, A.~M.~Gainutdinov and I.~Runkel, A quasi-Hopf algebra for the triplet vertex operator algebra,
  Comm. Contemp. Math. 2019, arXiv:1712.07260.

\bibitem{CHY}
T. Creutzig, Y.-Z. Huang and J. Yang, Braided tensor categories of admissible modules for affine Lie algebras, Comm. Math. Phys. {\bf 362} (2018), no. 3, 827--854.

\bibitem{CKL}
T.~Creutzig, S.~Kanade and A.~R.~Linshaw, Simple current extensions beyond semi-simplicity, Comm.\ Contemp.\ Math.\  (2019), 1950001.

\bibitem{CKLR}
 T.~Creutzig, S.~Kanade, A.~R.~Linshaw and D.~Ridout, Schur-Weyl Duality for Heisenberg Cosets, Transformation Groups {\bf 24} (2019), no. 2, 301-354.

\bibitem{CKM}
  T.~Creutzig, S.~Kanade and R.~McRae, Tensor categories for vertex operator superalgebra extensions,
  arXiv:1705.05017.

\bibitem{CKM2}
  T.~Creutzig, S.~Kanade and R.~McRae, Glueing vertex algebras,
  arXiv:1906.00119.

\bibitem{CG}
  T.~Creutzig and D.~Gaiotto, Vertex Algebras for S-duality,
  arXiv:1708.00875.

\bibitem{CGL}
  T.~Creutzig, D.~Gaiotto and A.~R.~Linshaw, S-duality for the large $N=4$ superconformal algebra, Comm. Math. Phys. (2020), arXiv:1804.09821.

  \bibitem{CG} T. Creutzig, T. Gannon, Logarithmic conformal field theory, log-modular tensor categories and modular forms, J. Phys. A {\bf 50} (2017), no. 40, 404004.

 \bibitem{CL}  T.~Creutzig and A.~R.~Linshaw, Cosets of affine vertex algebras inside larger structures,  J.\ Alg. {\bf 517} (2019), 396--438.



\bibitem{CR-log}
  T.~Creutzig and D.~Ridout, Logarithmic Conformal Field Theory: Beyond an Introduction,
  J.\ Phys.\ A {\bf 46} (2013), 494006, 72 pages.

\bibitem{CRW} T. Creutzig, D. Ridout, S. Wood, Coset constructions of logarithmic $(1,p)$-models, Lett. Math. Phys. {\bf 104} (2014), no. 5, 553--583.

\bibitem{CM}
  T.~Creutzig and A.~Milas, False Theta Functions and the Verlinde formula,
  Adv.\ Math.\  {\bf 262} (2014), 520--545.

\bibitem{CM2}
  T.~Creutzig and A.~Milas, Higher rank partial and false theta functions and representation theory,  Adv.\ Math.\  {\bf 314} (2017), 203--227.

\bibitem{CMW} T. Creutzig, A. Milas,  S. Wood:  On regularised quantum dimensions of the singlet vertex operator algebra and false theta functions, Int. Math. Res. Not. (2017), no. 5, 1390--1432.

  \bibitem{CS}
  C.~Cordova and S.~H.~Shao,
  Schur Indices, BPS Particles, and Argyres-Douglas Theories,
  JHEP {\bf 1601} (2016) 040.

\bibitem{CY} T. Creutzig and J. Yang, Tensor category of affine Lie algebras beyond admissble levels, in preparation.

\bibitem{CFJRY} T. Creutzig, C. Jiang, F. Orosz Hunziker, D. Ridout and J. Yang,  Tensor categories arising from the Virasoro algebra, in preparation.

\bibitem{DLM}
C. Dong, H. Li and G. Mason, Compact automorphism groups of vertex operator algebras, Int. Math.
Res. Not. (1996), 913--921.

\bibitem{DL1} C. Dong and J. Lepowsky, Abelian intertwining algebras--A generalization of vertex
operator algebras, in ``Algebraic Groups and Generalizations, Proc. 1991 Amer. Math.
Soc. Summer Research Institute (W. Haboush and B. Parshall, Eds.), Proc. Sympos.
Pure Math., Amer. Math. Soc., Providence, 1993.

\bibitem{DL2} C. Dong and J. Lepowsky, Generalized Vertex Algebras and Relative Vertex Operators, Progress in
Math., Vol. 112, Birkh$\ddot{\mbox{a}}$user, Boston, 1993.

\bibitem{FF90}B. L. Feigin, E. Frenkel,
\newblock Quantization of Drinfel'd-Sokolov reduction,
\newblock Phys. Lett., B {\bf 246} (1990), no. 1-2, 75--81.


\bibitem{FG}
  E.~Frenkel and D.~Gaiotto, Quantum Langlands dualities of boundary conditions, D-modules, and conformal blocks,
  arXiv:1805.00203.

\bibitem{FGST} B. L. Feigin, A.  Ga$\breve{\mbox{i}}$nutdinov, A. Semikhatov, I. Yu Tipunin: Modular group representations and fusion in logarithmic conformal field theories and in the quantum group center., Comm. Math. Phys {\bf 265} (2006), 47--93.

\bibitem{FGST2} B. L. Feigin, A.  Ga$\breve{\mbox{i}}$nutdinov, A. Semikhatov, I. Yu Tipunin: Kazhdan-Lusztig correspondence for the representation category of the triplet W-algebra in logarithmic CFT,
  Theor.\ Math.\ Phys.\  {\bf 148} (2006), no. 3, 1210--1235.

\bibitem{FT}
  B.~L.~Feigin and I. Yu Tipunin, Logarithmic CFTs connected with simple Lie algebras,  arXiv:1002.5047.

\bibitem{FMS}
D. Friedan, E. Martinec, S. Shenker, Conformal invariance, supersymmetry and string theory,
Nucl. Phys. B {\bf 271} (1986), 93--165.

\bibitem{GR}  D.~Gaiotto and M.~Rapcak, Vertex Algebras at the Corner,  JHEP {\bf 1901} (2019) 160.

\bibitem{H3}
Y.-Z. Huang,  Vertex operator algebras and the Verlinde conjecture, Comm. Contemp. Math. {\bf 10} (2008), 103--154.

\bibitem{H4}
Y.-Z. Huang, Rigidity and modularity of vertex tensor categories, Comm. Contemp. Math. {\bf 10} (2008), 871--911.

\bibitem{HKL}   Y.~Z.~Huang, A.~Kirillov and J.~Lepowsky, Braided tensor categories and extensions of vertex operator algebras,
  Comm.\ Math.\ Phys.\  {\bf 337} (2015) no.3,  1143--1159.

\bibitem{K}
S. Kumar, Extension of the category $\mathcal{O}^g$ and a vanishing theorem for the
Ext functor for Kac-Moody algebras, J. Alg., {\bf 108}(1987), no. 2, 472--491.

\bibitem{KFPX}  V. G. Kac, P. M. Frajria, P. Papi, F. Xu, Conformal embeddings and simple current extensions, Int. Math. Res. Not. (2015), no. 14, 5229-5288.

\bibitem{KL1} D. Kazhdan and G. Lusztig, Affine Lie algebras and quatum groups, International Mathematics Research Notices (in Duke Math. J.) {\bf 2} (1991), 21--29.

\bibitem{KL2}D. Kazhdan and G. Lusztig, Tensor structure arising from affine Lie algebras, I, J. Amer. Math. Soc. {\bf 6} (1993), 905--947.

\bibitem{KL3}D. Kazhdan and G. Lusztig, Tensor structure arising from affine Lie algebras, II, J. Amer. Math. Soc. {\bf 6} (1993), 949--1011.

\bibitem{KL4}D. Kazhdan and G. Lusztig, Tensor structure arising from affine Lie algebras, III, J. Amer. Math. Soc. {\bf 7} (1994), 335--381.

\bibitem{KL5}D. Kazhdan and G. Lusztig, Tensor structure arising from affine Lie algebras, IV, J. Amer. Math. Soc. {\bf 7} (1994), 383--453.

\bibitem{KaWi}
  A.~Kapustin and E.~Witten, Electric-Magnetic Duality And The Geometric Langlands Program,
  Comm.\ Num.\ Theor.\ Phys.\  {\bf 1} (2007), no. 1, 1--236.

\bibitem{KW3} V. Kac and M. Wakimoto, Quantum reduction and representation theory of superconformal algebras, Adv. Math. {\bf 185} (2004), 400--458.

\bibitem{M} R. McRae, On the tensor structure of modules for compact orbifold vertex operator algebras, arXiv: 1810.00747.

\bibitem{M2}
  R.~McRae, Twisted modules and $G$-equivariantization in logarithmic conformal field theory,
  arXiv:1910.13226.

\bibitem{Ras} L. Rastelli, Infinite Chiral Symmetry in Four and Six Dimensions, Seminar at Harvard University, November 2014.

\bibitem{TW}
  A.~Tsuchiya and S.~Wood, The tensor structure on the representation category of the $W_{p}$ triplet algebra,
  J.\ Phys.\ A {\bf 46} (2013), 445203.

\end{thebibliography}
\end{document}